\documentclass[letterpaper,11pt,reqno]{amsart}

\makeatletter
\usepackage{amssymb}
\usepackage{latexsym}
\usepackage{amsbsy}
\usepackage{amsfonts}
\usepackage{hyperref}
\usepackage{graphicx}
\usepackage{enumerate}
\usepackage{enumitem}
\usepackage{color}

\usepackage{tikz}
\usepackage{subcaption}
\usepackage[bottom]{footmisc}

\def\marginpar#1{\ignorespaces}

\DeclareMathOperator\argmin{argmin}
\DeclareMathOperator\Var{Var}

\DeclareMathOperator\rank{rank}
\DeclareMathOperator\sppa{span}

\newtheorem{theorem}{Theorem}[section]
\newtheorem{lemma}[theorem]{Lemma}
\newtheorem{proposition}[theorem]{Proposition}

\newtheorem{definition}[theorem]{Definition}

\newtheorem{assumpt}[theorem]{Assumption}
\numberwithin{equation}{section}
\makeatother

\begin{document}
\title[Existence of GLM MLE]{TThe Existence of Maximum Likelihood Estimate in 
High-Dimensional Binary Response Generalized Linear Models}

\author{Wenpin Tang}
\address{Department of Industrial Engineering and Operations Research, Columbia University.}
\email{wt2319@columbia.edu}

\author{Yuting Ye}
\address{Division of Biostatistics, UC Berkeley.}
\email{yeyt@berkeley.edu}

\date{\today} 
\begin{abstract}
Motivated by recent works on the high-dimensional logistic regression,
we establish that the existence of the maximum likelihood estimate exhibits a phase transition for a wide range of generalized linear models with binary outcome and elliptical covariates.
This extends a previous result of Cand\`es and Sur who proved the phase transition for the logistic regression with Gaussian covariates.
Our result reveals a rich structure in the phase transition phenomenon, which is simply overlooked by Gaussianity.
The main tools for deriving the result are data separation, convex geometry and stochastic approximation.
We also conduct simulation studies to corroborate our theoretical findings, and explore other features of the problem.
\end{abstract}

\maketitle

\section{Introduction}
In this paper, we are concerned with the maximum likelihood estimate of generalized linear models \cite{MN89, NW72} with binary outcome.
More precisely, we consider $n$ independent and identically distributed observations $(\pmb{x}_i, y_i)$, $i = 1, \ldots, n$, where the binary outcome $y_i \in \{-1,1\}$ is connected to the covariates $\pmb{x}_i \in \mathbb{R}^p$ by the probability model 
\begin{equation}
\label{eq:GLM}
\mathbb{P}(y_i = 1 | \pmb{x}_i) = \sigma(\beta_0 + \pmb{x}_i^T \pmb{\beta}).
\end{equation}
Here $\sigma: \mathbb{R} \rightarrow [0,1]$ is the {\em inverse link function}, and $(\beta_0, \pmb{\beta}) \in \mathbb{R}^{p+1}$ are unknown parameters of the model. 
Popular choices for $\sigma$ are
\begin{itemize}
\item
{\em logit link function} \cite{Ber44}: $\sigma(t) =e^t/(1+e^t)$.
\item
{\em probit link function} \cite{Bliss35}: $\sigma(t) = \Phi(t)$, where $\Phi(t): = \frac{1}{\sqrt{2 \pi}} \int_{-\infty}^t \exp(-s^2/2)ds$.
\item
{\em cloglog link function} \cite{Fisher}: $\sigma(t) = 1 - e^{-e^t}$.
\end{itemize}
The maximum likelihood estimate of $(\beta_0, \pmb{\beta})$ is any maximizer of the log-likelihood 
\begin{equation}
\label{eq:loglike}
\ell (\beta_0, \pmb{\beta}): = \sum_{y_i = 1} \log(\sigma(\beta_0 + \pmb{x}_i^T \pmb{\beta})) + \sum_{y_i = -1} \log(1- \sigma(\beta_0 + \pmb{x}_i^T \pmb{\beta})).
\end{equation}

In contrast to the maximum likelihood estimate of linear models, that of generalized linear models does not always exist. 
This phenomenon is closely related to the separability of observed data, see Section \ref{sc:existence} for a review.
Classical theory deals with this issue when the number of covariates $p$ is fixed, and the number of observed data $n$ tends to infinity.
In the era of data deluge, we are often in a situation where the number of covariates $p$ and the number of observations $n$ are comparable in size.
The problems of interest are in high-dimensional asymptotics, in which case the number of parameters $p$ and the number of observations $n$ both tend to infinity, at the same rate.

In a series of papers \cite{CS18, CS218, SCC17}, Sur, Chen and Cand\`es developed a theory for the logistic regression with Gaussian covariates in high-dimensional regimes.
They studied the asymptotic properties of the maximum likelihood estimate when $p/n \rightarrow \kappa$, with applications in hypothesis testing.
Cand\`es and Sur \cite{CS18} proved a phase transition for the existence of the maximum likelihood estimate in high-dimensional logistic regression with Gaussian covariates.
This extends an earlier result of Cover \cite{Cover65} in the context of information theory.
Formally, there exists a threshold $h_{\text{MLE}}$, depending on the parameters of the model, such that
\begin{itemize}
\item
if $\kappa > h_{\text{MLE}}$, then $\mathbb{P}({\mbox{maximum likelihood estimate exists}}) \rightarrow 0$ as $n, p \rightarrow \infty$.
\item
if $\kappa < h_{\text{MLE}}$, then $\mathbb{P}({\mbox{maximum likelihood estimate exists}}) \rightarrow 1$ as $n, p \rightarrow \infty$.
\end{itemize}
This phenomenon is referred to as the {\em phase transition} for the existence of the maximum likelihood estimate. 
The latter is crucial to justify the use of large sample approximations to numerous measures of goodness-of-fit, and derive the limiting distribution of the likelihood ratio, as mentioned in \cite{CS18}. 
But they only studied the existence of the maximum likelihood estimate for Gaussian covariates.
This is rarely the case in reality. 
For instance, the covariates are often heavy-tail distributed in financial problems where $p$ and $n$ are large.

The purpose of this paper is to further generalize the results of \cite{CS18}, proving the phase transition for a large class of generalized linear models with elliptical covariates. 
Here we consider a large number of covariates sampled from elliptical distributions, and predict whether one can expect the maximum likelihood estimate to be found or not. 
Elliptical symmetry is a natural generalization of multivariate normality.
The contribution of this paper is twofold.
\begin{itemize}[itemsep = 4pt]
\item \textbf{Theoretical justifications.} 
We give a universal threshold on $p/n$ for the existence of the maximum likelihood estimate in the binary classification. 
Here the word 'universal' refers to a wide class of link functions and covariate distributions. 
Our work aims to explore to which extent the phase transition occurs in terms of link functions and covariates, including the logit link and Gaussian covariates as a special case. 
We notice that the projection limit assumption (Assumption \ref{assumpt4}) is essential to our result, which is disguised for the special choice of Gaussian covariates. 
This is analogous to regularity assumptions in high-dimensional signal processing \cite{BM11, EK18}.
Without this assumption, the phase transition formula may fail (for example, log-normal covariates).

\item \textbf{Novel techniques.} 
While the high level idea is the same as \cite{CS18}, we bring a few techniques into this field.
First, we give a checkable condition to the projection limit assumption. 
Second, we use a stochastic approximation approach (Theorem \ref{thm:2}) to prove the phase transition formula. 
Compared to the bare-hands analysis in \cite{CS18}, our argument is more general and is easily adapted to other problems of interest. 
Finally, we provide fairly general conditions (for example, \eqref{eq:pG}) under which the phase transition occurs.
These conditions reveal additional structures which are masked by Gaussianity.
\end{itemize}

In a recent paper of Montanari, Ruan,  Sohn and Yan \cite{MR19}, they considered the max-margin classifier of the random feature problem in the high-dimensional regime.
They provided a phase transition threshold for the existence of the max-margin classifier, and further studied the limiting max-margin classifier. 
But similar to \cite{CS18}, they assume that the covariates are Gaussian.
De Loera and Hogan \cite{LH19} considered the maximum likelihood estimate of a multi-class logistic regression in a different way. 
They explored a condition on the number of observed data and the number of classes such that the maximum likelihood estimate exists.
We hope that our work will trigger further research towards a theory for multi-class classification models with non-Gaussian covariates.

The rest of the paper is organized as follows. 
In Section \ref{s2}, we provide the background and state the main result, Theorem \ref{thm:main}.
In Section \ref{s3}, we perform simulations to corroborate our theoretical findings. 
The proof of Theorem \ref{thm:main} is given in Section \ref{s4}. In Section \ref{s5}, we conclude with some insights and directions for future work.
%


\section{Background and Main Result}
\label{s2}
In this section we provide the background on the existence of the maximum likelihood estimate in binary response generalized linear models, and the properties of elliptical distributions.
Then we present the main result, Theorem \ref{thm:main}.
\subsection{Existence of the maximum likelihood estimate and Data Geometry}
\label{sc:existence}

Often the maximum likelihood estimate in the logit model, implemented in many statistical packages, runs smoothly.
But sometimes it fails, even when the number of covariates $p$ is much smaller than the sample size $n$.
One reason for this undesirable phenomenon is that the maximum likelihood estimate does not exist.
It is a classical problem in statistics to characterize the existence and uniqueness of the maximum likelihood estimate in generalized linear models.

Historically, Haberman \cite{Haber77} and Weddenburn \cite{W76} provided general criteria for the maximum likelihood estimate to exist.
Silvapulle \cite{Sil81}, and Albert and Anderson \cite{AA84} gave conditions for the existence of the maximum likelihood estimate in logistic regression via data geometry.
More precisely, they classified the data into the following three categories:

\begin{itemize}
\item
The data points $(\pmb{x}_i, y_i)$ are said to be {\em completely separated} if there exists $\pmb{b} \in \mathbb{R}^p$ such that 
$y_i \pmb{x}_i^T \pmb{b} > 0$ for all $i$.
\item
The data points $(\pmb{x}_i, y_i)$ are said to be {\em quasi-completely separated} if for each $\pmb{b} \ne \pmb{0}$, $y_i \pmb{x}_i^T \pmb{b}  \ge 0$ for all $i$, and equality holds for some $i$.
\item
The data points $(\pmb{x}_i, y_i)$ are said to {\em overlap} if for each $\pmb{b} \ne \pmb{0}$, there exists one $i$ such that $y_i \pmb{x}_i^T \pmb{b}  > 0$, and another $i$ such that $y_i \pmb{x}_i^T \pmb{b}  < 0$.
\end{itemize}

In \cite{AA84}, it was proved that the maximum likelihood estimate exists in logistic regression if and only if the data points overlap. 
See also \cite{SD86} for a generalization. 
Later Lesaffre and Kaufmann \cite{LK92} proposed a necessary and sufficient condition for the existence of the maximum likelihood estimate in probit regression, which coincides with that derived in \cite{AA84} for logistic regression.
In fact, their result holds for a general class of generalized linear models.

\begin{theorem}
\label{thm:LK}
Consider the generalized linear model defined by \eqref{eq:GLM},
and assume that $\sigma(\cdot)$ and $1 - \sigma(\cdot)$ are log-concave.
Then the maximum likelihood estimate exists if and only if the data points overlap.
\end{theorem}

It is easily seen that the logit, probit and cloglog links all satisfy the log-concavity. 
Despite this nice characterization, it is not clear how to check these criteria efficiently. 
See also \cite{CA01, Dem01, KK07} for algorithmic aspects for detecting separation/overlaps.
\subsection{Elliptical Distributions}
\label{sc:ED}
Elliptical distributions are natural generalizations of multivariate normal, which preserve spherical symmetry.
In the sequel, $\mathbb{S}^{p-1}$ denotes the unit sphere in $\mathbb{R}^p$.
The following definition of {\em elliptical distributions} is due to Kelker \cite{Kelker70}, and Cambanis, Huang and Simons \cite{CHS81}.

\begin{definition}
A random vector $\pmb{X} := (X_1, \ldots, X_p) \in \mathbb{R}^p$ is elliptically contoured, or simply elliptical if 
\begin{equation}
  \pmb{X} \stackrel{(d)}{=} \pmb{\mu} + R \, \pmb{A} \, \pmb{U},
  \label{eq:elliptical_dist}
\end{equation}
where $\pmb{\mu} \in \mathbb{R}^p$, $\pmb{A} \in \mathbb{R}^{p \times r}$, $\pmb{U}$ is uniformly distributed on $\mathbb{S}^{r-1}$ for some $r  > 0$, and $R$ is a non-negative random variable independent of $\pmb{U}$.
Write $\pmb{X} \sim \mathcal{E}_p(\pmb{\mu}, \pmb{\Sigma}, F_R)$, where 
$\pmb{\Sigma} = \pmb{A} \pmb{A}^T$, and $F_R$ is the cumulative distribution function of $R$.
\end{definition}

If $\pmb{\mu} = \pmb{0}$, $r = p$, and $\pmb{A}$ is orthogonal, then $\pmb{X} \sim \mathcal{E}_p(\pmb{0}, \pmb{I}_p, F_R)$ is said to be {\em spherically symmetric}. 
For $\pmb{X} \sim \mathcal{N}(\pmb{0}, \pmb{I}_p)$, the random variable $R$ is chi-distributed with degree of freedom $p$.
Below we list a few useful properties of elliptical distributions, see \cite{Bryc95, CHS81, FKN90} for further development.

\begin{enumerate}
\item
The random vector $\pmb{X}  \sim \mathcal{E}_p(\pmb{\mu}, \pmb{\Sigma}, F_R)$ has finite moments of order $k > 0$ if and only if 
$\mathbb{E}R^k < \infty$.  
If the first two moments exist, then
$\mathbb{E} \pmb{X} = \pmb{\mu}$ and $ \Var \pmb{X} = \frac{\mathbb{E}R^2}{r} \pmb{\Sigma}$.
\item
The distribution of $\pmb{X}  \sim \mathcal{E}_p(\pmb{\mu}, \pmb{\Sigma}, F_R)$ is absolutely continuous with respect to Lebesgue measure on $\mathbb{R}^p$ if and only if $r = p = \rank(\pmb{\Sigma})$, and the distribution of $R$ is absolutely continuous with respect to Lebesgue measure on $\mathbb{R}_{+}$.
\item
The marginal and conditional distributions of $\pmb{X}  \sim \mathcal{E}_p(\pmb{\mu}, \pmb{\Sigma}, F_R)$ are also elliptical. 
For the sake of simplicity, assume that $r = p = \rank(\pmb{\Sigma})$. 
Let $\pmb{X} := (\pmb{X}_1, \pmb{X_2})$, with $\pmb{X}_1 \in \mathbb{R}^{p_1}$ and $\pmb{X}_2 \in \mathbb{R}^{p_2}$ $(p_1 + p_2 = p)$. Let
$\pmb{\mu} := 
\begin{pmatrix}
    \pmb{\mu}_1 \\
   \pmb{\mu}_2
\end{pmatrix}$
and
$\pmb{\Sigma}: = 
\begin{pmatrix}
    \pmb{\Sigma}_{11} & \pmb{\Sigma}_{12} \\
   \pmb{\Sigma}_{21} & \pmb{\Sigma}_{22}
\end{pmatrix}$,
with $\pmb{\mu}_1 \in \mathbb{R}^{p_1}$, $\pmb{\mu}_2 \in \mathbb{R}^{p_2}$, $\pmb{\Sigma}_{11} \in \mathcal{M}_{p_1}(\mathbb{R})$, $\pmb{\Sigma}_{12} =\pmb{\Sigma}_{21}^T  \in \mathcal{M}_{p_1,p_2}(\mathbb{R})$, and $\pmb{\Sigma}_{22} \in \mathcal{M}_{p_2}(\mathbb{R})$.
Then $\pmb{X}_1 \sim \mathcal{E}_{p_1}(\pmb{\mu}_1, \pmb{\Sigma}_{11}, F_R)$, and $(\pmb{X}_1| \pmb{X}_2 = \pmb{x}_2) \sim \mathcal{E}_{p_1}(\pmb{\mu}_{1|2}, \pmb{\Sigma}_{1|2}, F_{R_{1|2}})$ where
\begin{equation*}
\pmb{\mu}_{1|2}: = \pmb{\mu}_1 + \pmb{\Sigma}_{12} \pmb{\Sigma}_{22}^{-1}(\pmb{x}_2 - \pmb{\mu}_2), \qquad 
\pmb{\Sigma}_{1|2}:= \pmb{\Sigma}_{11} - \pmb{\Sigma}_{12} \pmb{\Sigma}_{22}^{-1} \pmb{\Sigma}_{21}, 
\end{equation*}
and 
\begin{equation*}
\label{eq:condF}
F_{R_{1|2}}(r) = \frac{\int_{d_{\pmb{\Sigma}_{22}} (\pmb{x}_2, \pmb{\mu}_2)}^{\sqrt{r^2 +d^2_{\pmb{\Sigma}_{22}}(\pmb{x}_2, \pmb{\mu}_2)}} \left(s^2 -d^2_{\pmb{\Sigma}_{22}} (\pmb{x}_2, \pmb{\mu}_2)\right)^{p_1/2 - 1} s^{-p+2} dF_R(s)}{\int_{d_{\pmb{\Sigma}_{22}} (\pmb{x}_2, \pmb{\mu}_2)}^{\infty} \left(s^2 -d^2_{\pmb{\Sigma}_{22}} (\pmb{x}_2, \pmb{\mu}_2)\right)^{p_1/2 - 1} s^{-p+2} dF_R(s)},
\end{equation*}
with 
$d_{\pmb{\Sigma}_{22}} (\pmb{x}_2, \pmb{\mu}_2): = \sqrt{(\pmb{x}_2-\pmb{\mu}_2)^T \pmb{\Sigma}^{-1}_{22} (\pmb{x}_2-\pmb{\mu}_2)}$ the {\em Mahalanobis distance} between $\pmb{x}_2$ and $\pmb{\mu}_2$ in the metric associated with $\pmb{\Sigma}_{22}$.
\end{enumerate}

\subsection{Main Result}
\label{sc23}
Before stating the main result, we make a few assumptions on the link function, the covariate distribution, and model parameters.
As seen in Section \ref{sc:existence}, the existence of the maximum likelihood estimate of binary response generalized linear models can be translated into data geometry.
Thus, we need the assumptions on the link function in Theorem \ref{thm:LK}.

\begin{assumpt}[link function]
\label{assumpt1}
For $\sigma: \mathbb{R} \rightarrow [0,1]$, both $\sigma(\cdot)$ and $1 - \sigma(\cdot)$ are log-concave.
\end{assumpt}

The phase transition for the existence of the maximum likelihood estimate is expected to occur not only with Gaussian covariates, but with a broad range of covariate distributions.
Here we consider elliptical distributions.
To exclude singularity, we assume that the covariate distribution is absolutely continuous with respect to Lebesgue measure on $\mathbb{R}^p$. 

\begin{assumpt}[covariate distribution]
\label{assumpt2}
The covariates $\pmb{x}_i \sim \mathcal{E}_p(\pmb{\mu}, \pmb{\Sigma}, F_R)$ are of full rank. That is,
$r = p = \rank(\pmb{\Sigma})$ and $F_R$ is absolutely continuous on $\mathbb{R}_{+}$.
\end{assumpt}

To get a meaningful result in diverging dimension, we consider a sequence of problems with the intercept $\beta_0$ fixed, and $\Var(\pmb{x}_i^T \pmb{\beta}) \rightarrow \gamma_0^2$.
Recall that for $\pmb{x}_i \sim \mathcal{E}_p(\pmb{\mu}, \pmb{\Sigma}, F_R)$, we have
$\Var(\pmb{x}_i^T \pmb{\beta}) =  \frac{\mathbb{E}R^2}{p} \pmb{\beta}^T \pmb{\Sigma} \pmb{\beta}$.
This leads to the following assumption on model parameters.
\begin{assumpt}[parameter scaling]
\label{assumpt3}
As $p \rightarrow \infty$,
$\mathbb{E}R^2/p \rightarrow \alpha_0^2$ and  $|\pmb{\Sigma}^{1/2}\pmb{\beta}| \rightarrow \gamma_0 / \alpha_0$.
\end{assumpt}

In the remaining of the paper, we define
\begin{equation}
  (Y^{(p)},X^{(p)}) \sim F_{\alpha_0, \beta_0, \gamma_0} \quad \mbox{if} \quad (Y^{(p)},X^{(p)}) \stackrel{(d)}{=} (V^{(p)}, V^{(p)}U^{(p)}),
  \label{eq:projected_V_U}
\end{equation}
where $U^{(p)}$ is distributed as any component of $\pmb{x}_i \sim \mathcal{E}_p(\pmb{0}, \pmb{I}_p, F_R)$, and $\mathbb{P}(V^{(p)}=1|U^{(p)}) = 1 - \mathbb{P}(V^{(p)} = -1|U^{(p)}) = \sigma(\beta_0 + \gamma_0 U^{(p)}/\alpha_0)$.
The parametric distribution $F_{\alpha_0, \beta_0, \gamma_0}$ is an analog of $F_{\beta_0, \gamma_0}$ introduced in \cite{CS18}.
Here the superscript $(p)$ emphasizes that the distribution of $(Y^{(p)},X^{(p)})$ or $(V^{(p)},U^{(p)})$ may depend on $p$. 
The dependence in the Gaussian case is easily ignored since for $\pmb{x}_i \sim \mathcal{N}(\pmb{0}, \pmb{I}_p)$, $U^{(p)}$ is distributed as standard normal independent of $p$.
For the general elliptical covariates, we make the following technical assumption.
\begin{assumpt}[projection limit]
\label{assumpt4}
The projection $U^{(p)}$ converges in distribution to $U$. That way, 
$(Y^{(p)}, X^{(p)})$ converges in distribution to $(Y,X)$.
\end{assumpt}

Now we give a sufficient condition for Assumption \ref{assumpt4} to hold.
Let $m_{p,k}$ be the $k^{th}$ moment of $U^{(p)}$. 
If for each $k \ge 1$, $m_{p,k} \rightarrow m_k$ as $p \rightarrow \infty$ and 
\begin{equation}
\label{eq:Carl}
    \sum_{k = 1}^{\infty} m_{2k}^{-\frac{1}{2k}} = \infty,
\end{equation}
then $U^{(p)}$ converges in distribution to $U$ whose distribution is entirely characterized by the moments $(m_k; \, k \ge 1)$.
The condition \eqref{eq:Carl} is referred to as the {\em Carleman's condition}. 
See Lin \cite[Theorem 1]{Lin17} for a list of equivalent conditions. 

Define
\begin{equation}
\label{eq:p+-}
p_{+}(x) = \sigma\left(\beta_0 + \frac{\gamma_0}{\alpha_0}x\right) \quad \mbox{and} \quad p_{-}(x): = 1 - p_{+}(x).
\end{equation}
Denote $f_{X}$ as the density of $X$.
Also define
\begin{align}
\label{eq:G+-}
& G_{p,+}(x) = \int_{z \le x}p_{+}(z)f_{X^{(p)}}(z)dz   \quad \mbox{and} \quad G_{p,-}(x) = \int_{z \le x}p_{-}(z)f_{X^{(p)}}(z)dz, \notag \\
& \overline{G}_{p,+}(x) = \int_{z > x}p_{+}(z)f_{X^{(p)}}(z)dz \quad \mbox{and} \quad \overline{G}_{p,-}(x) = \int_{z > x}p_{-}(z)f_{X^{(p)}}(z)dz.
\end{align}
So $G_{p,+}(x) + G_{p,-}(x)$ is the cumulative distribution function of $X^{(p)}$, and $G_{p,\pm}(x) + \overline{G}_{p,\pm}(x) = \mathbb{E}p_{\pm}(X^{(p)})$.
The main result is stated as follows.

\begin{theorem}
\label{thm:main}
Let $(Y^{(p)},X^{(p)}) \sim F_{\alpha_0, \beta_0,\gamma_0}$, and $(Y, X)$ be the limit in distribution under Assumption \ref{assumpt4}.
Let $Z \sim \mathcal{N}(0,1)$ be independent of $(Y,X)$.
Define
\begin{align}
\label{eq:hmle}
h_{\text{MLE}}(\alpha_0, \beta_0,\gamma_0): &= \lim_{p \rightarrow \infty} \min_{\lambda_0, \lambda_1 \in \mathbb{R}} \mathbb{E}(\lambda_0 Y^{(p)} + \lambda_1 X^{(p)} - Z)_{+}^2  \notag \\
& =  \min_{\lambda_0, \lambda_1 \in \mathbb{R}} \mathbb{E}(\lambda_0 Y + \lambda_1 X - Z)_{+}^2,
\end{align}
where $x_{+}: = \max\{x,0\}$. 
If Assumptions \ref{assumpt1}-\ref{assumpt3} are satisfied and $\sup_p \mathbb{E}[(X^{(p)})^8]$ $< \infty$, and 
\begin{equation}
\label{eq:pG}
\mathbb{E}[p_{\pm}(X^{(p)}) (G_{p,\mp}(X^{(p)}) + \overline{G}_{p,\pm}(X^{(p)}))^{n-1}]  = o\left(\frac{1}{n}\right),
\end{equation}
then we have
\begin{align*}
& \kappa > h_{\text{MLE}}(\alpha_0, \beta_0, \gamma_0) \quad \Rightarrow \quad  \lim_{n,p \rightarrow \infty} \mathbb{P}({\mbox{maximum likelihood estimate exists}}) = 0. \\
& \kappa < h_{\text{MLE}}(\alpha_0, \beta_0, \gamma_0) \quad \Rightarrow \quad  \lim_{n,p \rightarrow \infty} \mathbb{P}({\mbox{maximum likelihood estimate exists}}) = 1.
\end{align*}
\end{theorem}

The proof is deferred to Section \ref{s4}. The assumption $\sup_p \mathbb{E}[(X^{(p)})^8] < \infty$, which is purely technical, is used to prove the law of large numbers for triangular arrays.
As suggested by simulations in Section \ref{s3}, the phase transition exists even without this moment condition.
The assumption \eqref{eq:pG} is used to prove that the probability the data points can be separated via a univariate model is small. 
A sufficient condition for \eqref{eq:pG} to hold is that $G_{p, \mp}+ \overline{G}_{p, \pm}$ is bounded away from $1$. 
That is, there exists $\epsilon > 0$ independent of $p$ such that
\begin{equation}
\label{eq:pGsuff}
G_{p,\mp}(x) + \overline{G}_{p,\pm}(x) < 1 - \epsilon \quad  \mbox{for all } x.
\end{equation}
So the term $\mathbb{E}[p_{\pm}(X^{(p)}) (G_{\mp}(X^{(p)}) + \overline{G}_{\pm}(X^{(p)}))^{n-1}]$ is exponentially small in $n$.
To illustrate, we check the condition \eqref{eq:pGsuff} for the logistic regression with Gaussian covariates. 
To simplify the discussion, we take $\beta_0 = 0$, and $\alpha_0 = \gamma_0 = 1$.
In this case, $p_{+}(x) = 1 - p_{-}(x) = e^x/(1+e^x)$ and $X^{(p)} = X \sim \mathcal{N}(0,1)$.
Fixing $x \ge 0$, we get
\begin{align*}
G_{p, -}(x) + \overline{G}_{p, +}(x) = \frac{1}{2} + \int_{x}^{\infty} \frac{e^z - 1}{1+e^z} f_X(z) dz < 1.
\end{align*}
Similarly, we can prove that $G_{p,-}(x) + \overline{G}_{p,+}(x) < 1$ for $x< 0$.
It is also easily checked that all examples in Section \ref{s3} except the log-normal distribution satisfy the sufficient conditions \eqref{eq:Carl}-\eqref{eq:pGsuff}.
\section{Empirical Results}
\label{s3}

In this section we perform experiments to verify the phase transition of the maximum likelihood estimate by $(i)$ computing $h_{\text{MLE}}$ defined by  \eqref{eq:hmle}; $(ii)$ checking whether the data is separated by the linear programming as in \cite{CS18}:
\begin{eqnarray}
  \max_{b_0, \pmb{b}} && \sum_{i = 1}^n y_i (b_0 + \pmb{x}_i^T \pmb{b}) \label{eq:separable_lp} \\
  \text{subject to} && y_i (b_0 + \pmb{x}_i^T \pmb{b}) \geq 0, \,  i = 1, \ldots, n\nonumber\\
                            && -1 \leq b_0 \leq 1, \, -\pmb{1} \leq \pmb{b} \leq \pmb{1}; \nonumber
\end{eqnarray}
Note that the maximum likelihood estimate of the binary response generalized linear model exists if the linear programming \eqref{eq:separable_lp} only has the trivial solution.
We compare the theoretical phase transition curve with the empirical observations under several simulation designs described as below.

First, we argue that $\pmb{A} = \pmb{I}_p$ suffices to validate our theory. 
Consider $\tilde{\pmb{X}} = \pmb{A}^{-1}\pmb{X}$ and $\tilde{\pmb{\beta}} = \pmb{A}^{\text{T}} \pmb{\beta}$ so that  $\tilde{\pmb{X}}\tilde{\pmb{\beta}} = \pmb{X}\pmb{\beta}$. 
The conditional distribution of the response $Y$ given $\tilde{\pmb{X}}\tilde{\pmb{\beta}}$ is the same as that given $\pmb{X}\pmb{\beta}$. If $(\tilde{\pmb{X}}_1, Y_1), \ldots, (\tilde{\pmb{X}}_n, Y_n)$ are linearly separated by a hyperplane $\pmb{\beta}^*$, then $(\pmb{X}_1, Y_1), \ldots, (\pmb{X}_n, Y_n)$ are linearly separated by a hyperplane $(\pmb{A}^{\text{T}})^{-1}\pmb{\beta}^*$. The other way around also holds. 
Note that $\tilde{\pmb{X}} \sim \mathcal{E}_p (\pmb{A}^{-1}\pmb{\mu},\pmb{I}_p, F_R)$, so it suffices to consider $\pmb{A} = \pmb{I}_p$.

For the elliptical covariates, we set $\pmb{\mu} = \bold 0$, $\pmb{A} = \pmb{I}_p$ and consider different distributions for the non-negative variable $R$, including the chi distribution with degree of freedom $p$, the Gamma distributions, the Pareto distributions, the half-normal distribution and the log-normal distribution. When generating the binary response by \eqref{eq:GLM}, we choose the logit function, the cloglog function and the probit function for the link functino. We simply take $\beta_0 = 0$. See \cite{CS18} for results with $\beta_0 \neq 0$.
To ensure Assumption \ref{assumpt3}, let
\begin{equation}
    \mathbb{E} R^2 = p \alpha_0^2 + 1 \quad \mbox{and} \quad \pmb{\beta} = (\widetilde{\bf W}/||\widetilde{\bf W}||_2 + 1/p) \cdot \gamma_0/\alpha_0,
    \label{eq:R2_beta}
\end{equation}
where $\widetilde{\bf{W}} \sim \mathcal{N}(\pmb{0}, \pmb{I}_p)$.

We fix $n = 1000$, $\alpha_0 = 1$, and vary $\gamma_0 \in \{0.01, 0.02, \ldots, 10.00\}$ and $\kappa = p/n \in \{0.005, 0.01, \ldots, 0.6\}$. The parameter $\alpha_0$ is simply set as 1 since we observed that it does not affect the phase transition much. A large $\alpha_0$ might slightly shift the phase transition curve (the reds curve in Figure \ref{fig:pt_chip_link}) to the right, and enlarge the uncertain band (the green bands). Once the data is generated, we solve the problem \eqref{eq:separable_lp} by checking whether a non-trivial solution exists.
We repeat the procedure for $100$ times, and get a heat map which indicates the proportion of times that the maximum likelihood estimate exists for each pair $(\gamma_0, \kappa)$.
See Figure \ref{fig:pt_chip_link} for the chi case with degree of freedom $p$. 
Results for other designs can be found in Appendix \ref{appendix:results}.

\subsection{Multivariate Gaussian covariates with different link functions}\label{sec:sim_multigaussian}
We consider the multivariate Gaussian covariates, which have been studied for the logit link in \cite{CS18}.
In our setup, $R$ is sampled from a chi distribution with degree of freedom $p$ and the link function is one of \{\textit{logit}, \textit{cloglog}, \textit{probit}\}.
Figure \ref{fig:pt_chip_link} displays the phase transition for the existence of the maximum likelihood estimate for different link functions.
There are green bands in these figures, which indicates that the maximum likelihood estimate exists indefinitely when $(\gamma_0, \kappa)$ falls in this band with the given sample size.
This region is referred to as the {\em uncertainty band}.
Observe that for the multivariate Gaussian covariates, as expected, the $h_{\text{MLE}}$ curves lie in the uncertainty bands for different link functions. 
\begin{figure}[ht]
  \centering
  \includegraphics[width=1 \textwidth]{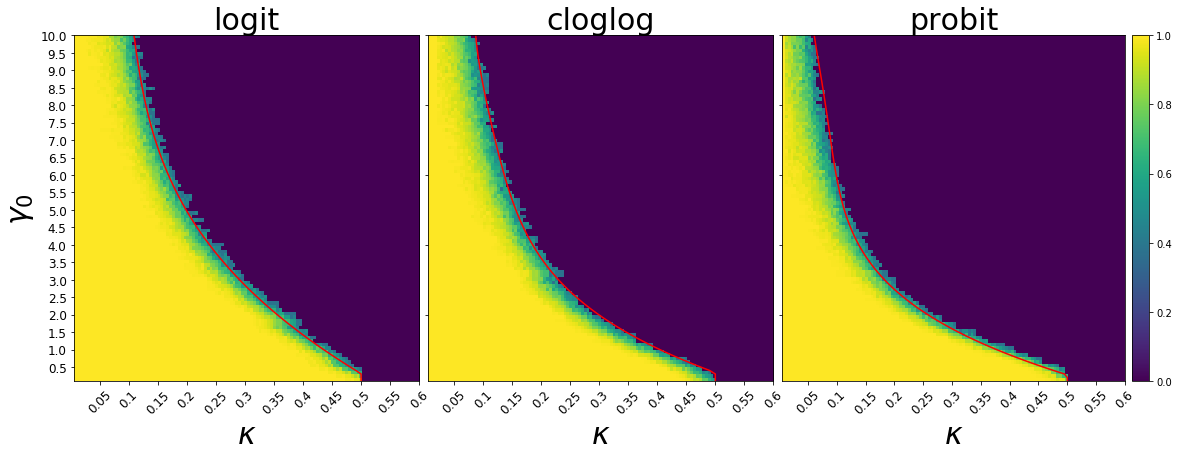}
  \caption{Phase transition for the existence of the maximum likelihood estimate for $R \sim$ chi distribution with degree of freedom $p$. 
  The red curve is the theoretical $h_{\text{MLE}}$ boundary given by \eqref{eq:hmle}.}
  \label{fig:pt_chip_link}  
\end{figure}

We also use the same setup to investigate the situation when $\pmb{A} \neq \pmb{I}_p$. In this case, we generate $\pmb{A}$ such that each entry $A_{ij}$ is i.i.d sampled from the standard Gaussian distribution. To make sure it is full rank, we let $\pmb{A} \leftarrow \frac{1}{10000}\Vert \pmb{A} \Vert_{2} \cdot \pmb{I}_p + \pmb{A}$. The result is deferred to Figure \ref{fig:pt_chip_link_full} (Middle). We find that the result is very similar to that in Figure \ref{fig:pt_chip_link}, which corroborates our argument that it suffices to use $\pmb{A} = \pmb{I}_p$ to validate our theory.

\subsection{Gamma-distributed $R$}\label{sec:sim_gamma}
In \cite{CS18}, $U^{(p)}$ defined by \eqref{eq:projected_V_U} simplifies to the Gaussian distribution, which does not depend on $p$. 
This is key to their proof of the phase transition for the existence of the maximum likelihood estimate. 
However, when we go beyond the chi distribution for $R$, $U^{(p)}$ depends on $p$.
We observe that Assumption \ref{assumpt4} is satisfied for Gamma distributions, and the resulting theoretical phase transition curves agree with the simulations.
More precisely, assume $R \sim \text{Gamma}(k, \theta)$ where $k$ is the shape parameter and $\theta$ is the scale parameter. 
The second moment condition \eqref{eq:R2_beta} gives $\theta = \sqrt{\mathbb{E}R^2/(k^2 + k)}$.
When $k = 0.5$, we get $\theta_{0.5} = \sqrt{4(p+1)/3}$ which corresponds to $\chi$ distribution with degree of freedom $2$ if $\theta_{0.5}$ is an integer; when $k = 1$, it is the Exponential distribution with $\theta_1 = \sqrt{(p+1)/2}$; when $k = 2$, it is a Gamma distribution with $\theta_2 = \sqrt{(p+1)/6}$. 
Figure \ref{fig:pt_gamma_hmle_conv} implies that $h_{\text{MLE}}$ defined by \eqref{eq:hmle} converges quickly as $p$ increases. 
Table \ref{tbl:hmle_bw_tv} indicates that all the theoretical phase transition curves align with the corresponding middle curves of the uncertainty bands.

\begin{figure}[ht]
  \centering
  \includegraphics[width=  \textwidth]{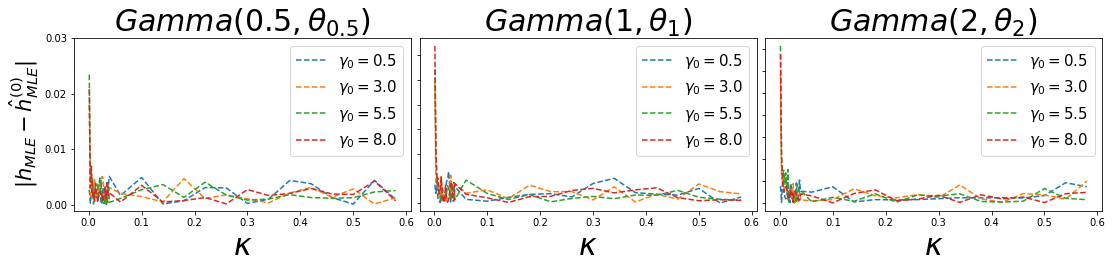}    
  \caption{Convergence of $h_{\text{MLE}}$ by \eqref{eq:hmle} for Gamma distributions. $\hat{h}_{MLE}^{(0)}$ is computed by taking the average of $h_{\text{MLE}}$ for $\kappa \geq 0.3$.}
  \label{fig:pt_gamma_hmle_conv}  
\end{figure}

\begin{table}[]
  \centering
  \caption{Summary of the theoretical $h_{\text{MLE}}$ and the simulations of the phase transition for the maximum likelihood estimate with the logit link function. $h_{0.5}$ is the $\kappa$ such that the proportion of times that the maximum likelihood estimate exists is $0.5$, and the number inside the bracket is the width of the uncertainty interval (a slice of the \textit{uncertainty band})  for a given $\gamma_0$. MIW is the mean width of the uncertainty intervals across $\gamma_0 \in (0, 10]$; MD is the mean difference between $h_{\text{MLE}}$ and $h_{0.5}$.}
  \label{tbl:hmle_bw_tv}
\begin{tabular}{c|cc|cc|cc}
\hline
                           & \multicolumn{2}{c|}{$\gamma_0 = 1$} & \multicolumn{2}{c|}{$\gamma_0 = 9$} & \multicolumn{2}{c}{Overall} \\ \hline
Distribution               & $h_{0.5}$       & $h_{\text{MLE}}$      & $h_{0.5}$       & $h_{\text{MLE}}$      & MIW           & MD            \\ \hline
$Gamma(0.5, \theta_{0.5})$ & 0.435 (0.075)      & 0.4238         & 0.310 (0.095)      & 0.310          & 0.101        & 0.0077        \\ 
$Gamma(1, \theta_{1})$     & 0.450 (0.050)      & 0.458          & 0.260 (0.080)      & 0.246          & 0.071        & 0.0186        \\ 
$Gamma(2, \theta_{2})$     & 0.450 (0.050)      & 0.447          & 0.205 (0.070)      & 0.191          & 0.057        & 0.0160        \\ \hline
$Pareto(2.5, x_m(2.5))$    & 0.455 (0.045)      & 0.458          & 0.165 (0.065)      & 0.172          & 0.055        & 0.0045        \\ 
$Pareto(3.5, x_m(3.5))$    & 0.440 (0.045)      & 0.439          & 0.120 (0.060)      & 0.137          & 0.057        & 0.0095        \\ 
$Pareto(4.5, x_m(4.5))$    & 0.435 (0.055)      & 0.430          & 0.110 (0.060)      & 0.128          & 0.055        & 0.0095        \\ \hline
half-normal                & 0.450 (0.050)      & 0.448          & 0.240 (0.075)      & 0.211          & 0.064        & 0.0215        \\ \hline
log-normal                 & 0.380 (0.135)      & 0.497          & 0.355 (0.150)      & 0.450          & 0.152        & 0.1199        \\ \hline
\end{tabular}
\end{table}
  



\subsection{The moment condition and the tail behavior of $R$}\label{sec:sim_moment_tail}
First we explore a case where the eighth moment of $X^{(p)}$ does not exist as required by Theorem \ref{thm:main}. 
To this end, we sample $R$ from the Pareto distribution of type I. 
We specify the shape parameter $\alpha$, and set the scale parameter $x_m = \sqrt{(\alpha - 2)/\alpha \cdot \mathbb{E} R^2}$. 
Recall that for the Pareto distribution of type I, the fourth moment exists when $\alpha > 4$, and the third moment exists when $\alpha > 3$. 
From Table \ref{tbl:hmle_bw_tv}, we see that the simulation results match the theoretical $h_{\text{MLE}}$ well. 
This suggests that the moment condition in Theorem \ref{thm:main} may be further relaxed.

Subsequently, we study how the tail behavior of the $R$ distribution influences the phase transition curve $h_{\text{MLE}}$. 
In the previous empirical studies, we consider the chi distributions with degree of freedom $p$ and Gamma distributions which have sub-exponential tails; the Pareto distributions have polynomial tails. 
We also investigate the half-normal distribution with a sub-Gaussian tail, and the log-normal distribution with another heavy tail. 
To ensure \eqref{eq:R2_beta}, we set the scale parameter $\sigma^2 = \mathbb{E} R^2 = p + 1$ for the half-normal distribution, and $\sigma^2 = 0.5 \cdot \log \mathbb{E} R^2 = \log \sqrt{p+1}$ for the log-normal distribution. 
From Table \ref{tbl:hmle_bw_tv}, we observe that the theoretical $h_{\text{MLE}}$ successfully predicts the phase transition in the simulations of the half-normal distribution, but fails for the log-normal distribution. 
This is due to the fact that the log-normal distribution is not uniquely characterized by its moments, and the sufficient condition \eqref{eq:Carl} does not hold. 
See Appendix \ref{appendix:results} for more detailed results for the exploration.


\section{Proof of Theorem \ref{thm:main}}
\label{s4}

\subsection{Roadmap to the proof of Theorem \ref{thm:main}}

\fbox{{\em Elliptical covariates}}
Assume that the covariates $\pmb{x}_i \sim \mathcal{E}_p(\pmb{\mu}, \pmb{\Sigma}, F_R)$ with $r = p = \rank(\pmb{\Sigma})$.
It is easily seen that $\pmb{x}_i = \pmb{\mu} + \pmb{\Sigma}^{1/2}\pmb{z}_i$ with $\pmb{z}_i \sim \mathcal{E}_p(\pmb{0}, \pmb{I}_p, F_R)$. 
Recall from Section \ref{sc:existence} that for GLMs satisfying Assumption \ref{assumpt1}, 
the MLE does not exist if and only if there is $(b_0, \pmb{b}) \ne \pmb{0}$ such that 
$y_i(b_0 + \pmb{x}_i^T \pmb{b}) \ge 0$ for all $i$. 
This is equivalent to the existence of $(\widetilde{b}_0, \widetilde{\pmb{b}}) \ne \pmb{0}$ such that $y_i(\widetilde{b}_0 + \pmb{z}_i^T \widetilde{\pmb{b}}) \ge 0$ for all $i$.
Without loss of generality, we assume $\pmb{x}_i \sim \mathcal{E}_p(\pmb{0}, \pmb{I}_p, F_R)$ in the sequel.

We are in a situation where the covariates $\pmb{x}_i: = (x_{i1}, \ldots, x_{ip})$ is spherically symmetric and $\Var(\pmb{x}_i^T \pmb{\beta}) = |\pmb{\beta}|^2 \mathbb{E}R^2 / p \rightarrow \gamma_0^2$.
By rotational invariance, we assume that all the signal is in the first coordinate.
That is,
$\mathbb{P}(y_i = 1| \pmb{x}_i) = \sigma \left(\beta_0 + \frac{\gamma_0}{\alpha_0} x_{i1}\right)$.
The results in Section \ref{sc:ED} show that
\begin{equation}
\label{eq:joint}
(y_i, y_i\pmb{x}_i) \stackrel{(d)}{=} (Y^{(p)}, X^{(p)}, X_2, \ldots, X_p),
\end{equation} 
where $(Y^{(p)},X^{(p)}) \sim F_{\alpha_0, \beta_0, \gamma_0}$, and 
$$(X_2, \ldots, X_p|X^{(p)} = x) \sim \mathcal{E}_{p-1}(\pmb{0}, \pmb{I}_{p-1}, F_{R_{-1}}),$$ 
with
$F_{R_{-1}}(r) = \frac{\int_{|x|}^{\sqrt{r^2 + x^2}}(s^2 - x^2)^{(p-3)/2} s^{-p+2} dF_R(s)}{\int_{|x|}^{\infty}(s^2 - x^2)^{(p-3)/2} s^{-p+2} dF_R(s)}$.

Now we want to express $\mathbb{P}(\mbox{no MLE})$ via conic geometry. 
For a fixed space $\mathcal{W} \in \mathbb{R}^n$, let 
\begin{equation*}
\mathcal{C}(\mathcal{W}) = \{\pmb{w} + \pmb{u}: \pmb{w} \in \mathcal{W}, \, \pmb{u} \ge 0\},
\end{equation*}
be the {\em convex cone} generated by $\mathcal{W}$.
The following proposition is read from \cite[Propositions 1 $\&$ 2]{CS18}, 
and we include the proof in Section \ref{B1} for completeness.

\begin{proposition}
\label{prop:41}
Let the $n$-dimensional vectors $(\pmb{Y}^{(p)}, \pmb{X}^{(p)}, \pmb{X}_2, \ldots , \pmb{X}_p)$ be $n$ i.i.d. copies of $(Y^{(p)},X^{(p)},X_2, \ldots, X_p)$ distributed as in \eqref{eq:joint}.
Let 

$$\mathcal{L}: = \sppa(\pmb{X}_2, \ldots, \pmb{X}_p) \quad \mbox{and} \quad \mathcal{W}: = \sppa(\pmb{Y}^{(p)}, \pmb{X}^{(p)}).$$
Let $\{\mbox{No MLE Single}\}$ be the event that the data points can be completely or quasi-completely separated by 
the intercept and the first coordinate only, i.e. $\mathcal{W} \cap \mathbb{R}_{+}^n \ne \{\pmb{0}\}$.
Then 
\begin{equation}
0 \leq \mathbb{P}(\mbox{no MLE}) - \mathbb{P}(\mathcal{L} \cap \mathcal{C}(\mathcal{W}) \ne \{ \pmb{0} \}) \le \mathbb{P}(\mbox{No MLE Single}).
\end{equation}
\end{proposition}

By Proposition \ref{prop:41}, the existence of the MLE boils down to whether $\mathcal{L}$ intersects $\mathcal{C}(\mathcal{W})$ in a non-trivial way. 
It remains to prove the following: 
$(1)$ The probability $\mathbb{P}(\mbox{No MLE Single})$ is relatively small.
$(2)$ The probability $\mathbb{P}(\mathcal{L} \cap \mathcal{C}(\mathcal{W}) \ne \{ \pmb{0} \})$ exhibits a phase transition through the ratio $\kappa: = p/n$, and $h_{\text{MLE}}(\alpha_0,\beta_0,\gamma_0)$ defined by \eqref{eq:hmle}.

\smallskip
\noindent
\fbox{{\em Separation of data in a univariate model}}
We aim to prove that the probability $\mathbb{P}(\mbox{No MLE Single})$ is small. 
In \cite{CS18}, a sketch of proof is given for the logistic regression with Gaussian covariates. 
In Section \ref{B2}, we give a rigorous proof of this result in the setting of Theorem \ref{thm:main}.
The main difficulty comes from the fact that though the probability the data can be separated via any fixed $t_0 \in \mathbb{R}$ is exponentially small, there are uncountably many such $t_0$ and the union bound does not give a good estimate.

\begin{proposition}
  \label{prop2}
Under the assumptions in Theorem \ref{thm:main}, the event $\{\mbox{No MLE}$ $\mbox{Single}\}$ occurs with small probability.
That is,
$\mathbb{P}(\mbox{No MLE Single}) = o(1)$.
\end{proposition}

\noindent
\fbox{{\em Convex geometry and phase transition}} 
We want to prove the phase transition of $\mathbb{P}(\mathcal{L} \cap \mathcal{C}(\mathcal{W}))$ through the interplay between $\kappa$ and $h_{\text{MLE}}(\alpha_0, \beta_0,\gamma_0)$.
The key is to understand when a random subspace $\mathcal{L}$ with uniform orientation intersects $\mathcal{C}(\mathcal{W})$ in a non-trivial way.

For any fixed subspace $\mathcal{W} \in \mathbb{R}^n$, the {\em approximate kinematic formula} \cite[Theorem I]{ALMT}
shows that for any $\varepsilon \in (0,1)$, there exists $a_{\varepsilon} > 0$ such that
\begin{align}
\label{eq:geo}
& p - 1 + \delta(\mathcal{C}(\mathcal{W})) > n + a_{\varepsilon} \sqrt{n} \Longrightarrow \mathbb{P}(\mathcal{L} \cap \mathcal{C}(\mathcal{W})) \ge 1 - \varepsilon, \notag\\
& p - 1 + \delta(\mathcal{C}(\mathcal{W})) < n - a_{\varepsilon} \sqrt{n} \Longrightarrow \mathbb{P}(\mathcal{L} \cap \mathcal{C}(\mathcal{W})) \le \varepsilon,
\end{align}
Here $\delta(\mathcal{C})$ is the {\em statistical dimension} of the convex cone $\mathcal{C}$ defined by
$\delta(\mathcal{C}) := n - \mathbb{E}|\pmb{Z} - \Pi_{\mathcal{C}} (\pmb{Z})|^2$, where $\pmb{Z} \sim \mathcal{N}(\pmb{0}, \pmb{I}_n)$ and $\Pi_{\mathcal{C}}$ is the projection onto $\mathcal{C}$.
The following identity is given in \cite[Lemma 3]{CS18}:
\begin{equation}
\label{eq:SD}
\delta(\mathcal{C}(\mathcal{W})) = n - \mathbb{E}\left( \min_{\pmb{w} \in \mathcal{W}} |(\pmb{w} - \pmb{Z})_{+}|^2 \right).
\end{equation}

Theorem \ref{thm:main} can be derived from the formulas \eqref{eq:geo}-\eqref{eq:SD} and the following theorem.
 \begin{theorem}
\label{thm:2}
Let $(\pmb{Y}^{(p)}, \pmb{X}^{(p)})$ be $n$ i.i.d. samples from $F_{\alpha_0, \beta_0, \gamma_0}$ satisfying Assumption \ref{assumpt4},
and $\sup_p \mathbb{E}[(X^{(p)})^8] < \infty$. 
Let
$$Q_{p,n}: = \min_{\lambda_0, \lambda_1 \in \mathbb{R}} \frac{1}{n} |(\lambda_0 \pmb{Y}^{(p)} + \lambda_1 \pmb{X}^{(p)} - \pmb{Z})_{+}|^2.$$
Then $Q_{p,n}$ converges in probability to $\min_{\lambda_0, \lambda_1 \in \mathbb{R}} \mathbb{E}(\lambda_0 Y + \lambda_1 X - Z)_{+}^2$ as $n,p \rightarrow \infty$.
\end{theorem}

In \cite{CS18}, the authors proved Theorem \ref{thm:2} in the setting of the logistic regression by a bare-hands argument. 
One can adapt their argument to prove Theorem \ref{thm:2}, with possibly more complications.
However, the statement of Theorem \ref{thm:2} suggests it be a form of stochastic approximation.
Here we show how this result follows systematically from stochastic approximation, which is of independent interest.

\smallskip
\noindent
\fbox{{\em Stochastic approximation}}
We sketch a proof of Theorem \ref{thm:2} via a stochastic approximation. 
In the stochastic approximation literature \cite{KW52, RM51}, people seek to approximate the optimization problem
\begin{equation}
\label{eq:opt}
\min_{\pmb{\lambda} \in S} G(\pmb{\lambda}), \quad G(\pmb{\lambda}): = \mathbb{E} g(\pmb{\lambda}, \pmb{\xi}),
\end{equation}
where $S \subset \mathbb{R}^k$ for some $k > 0$ and $\pmb{\xi}$ is a generic random vector, by a sequence of stochastic optimization problems 
$\min_{\pmb{\lambda} \in S} \widehat{G}_n(\pmb{\lambda}; \pmb{\xi}_1, \ldots, \pmb{\xi}_n)$,  $\widehat{G}_n(\pmb{\lambda}; \pmb{\xi}_1, \ldots, \pmb{\xi}_n): = \frac{1}{n} \sum_{i = 1}^n g(\pmb{\lambda}, \pmb{\xi}_i)
$
where $(\pmb{\xi}_i; \, i \ge 1)$ are i.i.d. copies of $\pmb{\xi}$.

A deep connection between stochastic approximation and convergence of random closed sets was established by Attouch and Wets \cite{AW81} via the {\em epi-convergence} of functions. 
A sequence of lower semi-continuous functions $f_n: \mathbb{R}^{k} \rightarrow (\infty, \infty]$ is said to epi-converges to $f$ if for each $x \in \mathbb{R}^k$,
\begin{itemize}
\item
$\lim \inf f_n(x_n) \rightarrow f(x)$ if $x_n \rightarrow x$,
\item
$\lim  f_n(x_n) \rightarrow f(x)$ for at least one sequence $x_n \rightarrow x$.
\end{itemize}
See \cite{AR95, Attoch84, DW88, Hess96, KW91, Sh91} for further development on epi-convergence. 

 Here we consider a sequence of stochastic optimization problems with triangular arrays
\begin{equation}
\label{eq:stopt2}
\min_{\pmb{\lambda} \in S} \widehat{G}_n(\pmb{\lambda}; \pmb{\xi}_{1,n}, \ldots, \pmb{\xi}_{n,n}), \quad \widehat{G}_n(\pmb{\lambda}; \pmb{\xi}_{1,n}, \ldots, \pmb{\xi}_{n,n}): = \frac{1}{n} \sum_{i = 1}^n g(\pmb{\lambda}, \pmb{\xi}_{i,n}),
\end{equation}
where $(\pmb{\xi}_{i,n}; \, 1 \le i \le n)$ are i.i.d. copies of $\pmb{\xi}_n$, and $\pmb{\xi}_n$ converges in distribution to $\pmb{\xi}$.
Let $v$, $\widehat{v}_n$, and $\argmin G$, $\argmin \widehat{G}_n$ be optimal values, and optimal solutions to the problems \eqref{eq:opt}-\eqref{eq:stopt2}.
Note that $\argmin \, G$ and $\argmin \, \widehat{G}_n$ are set-valued.
The following result gives asymptotic inference of $\widehat{v}_n$ as $n \rightarrow \infty$.
The proof will be given in Section \ref{B3}.

\begin{lemma}
\label{lem:sa}
Assume that 
$g(\cdot, \cdot)$ is measurable and bounded from below, and $g(\cdot, \pmb{\xi})$ is convex.
Assume that $\sup_n \mathbb{E}g^4(\pmb{\lambda}, \pmb{\xi}_{n}) < \infty$ and $\mathbb{E}g(\pmb{\lambda}, \pmb{\xi}_n) \rightarrow \mathbb{E}g(\pmb{\lambda}, \pmb{\xi})$ for all $\pmb{\lambda}$.
Further assume that $\argmin \, \widehat{G}_n$, $n \ge 1$ are non-empty and bounded in probability.
Then $\argmin \, G \ne \emptyset$, and 
\begin{equation*}
\widehat{v}_n \rightarrow v \quad \mbox{in probability}.
\end{equation*}
\end{lemma}

To prove Theorem \ref{thm:2}, we need to show that the set of minimizers $\argmin \, \widehat{G}_n$ is non-empty and bounded in probability.
The argument is similar in spirit to \cite{CS18}, and we give the proof for ease of reference.
In Section \ref{B3},  we prove that under the assumptions in Theorem \ref{thm:main}:
\begin{itemize}
\item
The problem \eqref{eq:opt} has a unique minimizer $\pmb{\lambda}_0$.
\item
For any minimizer $\widehat{\pmb{\lambda}}_n \in \argmin \, \widehat{G}_n$, $|\widehat{\pmb{\lambda}}_n - \pmb{\lambda}_0| = \mathcal{O}_P(1)$.
Here we use the Chebyshev inequality instead of the concentration inequality for sub-exponential variables in \cite{CS18}.
\end{itemize}

\smallskip
\noindent
\fbox{\em Proof of Theorem \ref{thm:main}}
The proof goes along the same line as \cite{CS18}, with two modifications. 
Assume that $\kappa > h_{\text{MLE}}(\alpha_0, \beta_0,\gamma_0)$.
\begin{itemize}
\item
Given $(Y^{(p)}, X^{(p)})$, the random vector $(X_2, \ldots, X_p)$ is also elliptical whose distribution is given by \eqref{eq:joint}. 
By the geometric characterization \eqref{eq:geo}, we get
\begin{equation*}
\mathbb{P}(\mathcal{L} \cap \mathcal{C}(\mathcal{W}) \ne \{\pmb{0}\}) \ge \mathbb{P}(p/n > \mathbb{E}(Q_{p.n}|\pmb{X}, \pmb{Y} )+ a_n n^{-1/2} ) - \varepsilon_n,
\end{equation*}
for some $\varepsilon_n \rightarrow 0$.
\item
The random variable $Q_{p,n}$ is uniformly integrable, since $Q_{p,n} \le |\pmb{Z}_{+}|^2/n$ which is sub-exponential. 
This implies that $\mathbb{E}(Q_{p,n}|\pmb{X}, \pmb{Y} )$ converges in probability to $h_{\text{MLE}}(\alpha_0, \beta_0,\gamma_0)$.
\end{itemize}
Thus, $\mathbb{P}(\mathcal{L} \cap \mathcal{C}(\mathcal{W}) \ne \{\pmb{0}\}) \rightarrow 1$ and $\mathbb{P}(\mbox{MLE exists}) \rightarrow 0$. 
Similarly, we can prove if $\kappa < h_{\text{MLE}}(\alpha_0, \beta_0,\gamma_0)$, then $\mathbb{P}(\mbox{MLE exists}) \rightarrow 1$. 

\subsection{Proof of Proposition \ref{prop:41}} \label{B1}
We aim to prove that 
\begin{equation}
\label{eq:star1}
\mathbb{P}(\mbox{No MLE}) = \mathbb{P}(\mbox{No MLE Single}) + \mathbb{P}(\mathcal{L} \cap \mathcal{C}(\mathcal{W}) \ne \{\pmb{0}\} \mbox{ and } \{\mbox{No MLE Single}\}^c),
\end{equation}
from which the result follows.
If $\{\mbox{No MLE Single}\}$ occurs, there is no MLE. 
Assume that $\{\mbox{No MLE Single}\}$ does not occur. 
If
\begin{equation}
\label{eq:star2}
\mathbb{P}(\mbox{no MLE}) = \mathbb{P}(\mbox{Span}(\pmb{Y}^{(p)}, \pmb{X}^{(p)}, \pmb{X}_2, \ldots \pmb{X}_p) \cap \mathbb{R}_{+}^n \ne \{\pmb{0}\}),
\end{equation}
then there is no MLE if and only if there is a non-zero vector $(b_0, \ldots, b_p)$ such that
$b_0\pmb{Y}^{(p)} + b_1 \pmb{X}^{(p)} + \cdots + b_p \pmb{X}_p = \pmb{u}$, $\pmb{u} \ge 0, \pmb{u} \ne 0$.
By assumption, $b_0\pmb{Y}^{(p)} + b_1 \pmb{X}^{(p)} \ne u$ so $b_2 \pmb{X}_2 + \cdots b_p \pmb{X}_p$ is a non-zero element of $\mathcal{C}(\mathcal{W})$. This leads to \eqref{eq:star1}.
Note that there is no MLE if and only if there is a non-zero vector $(b_0, \ldots, b_p)$ such that
$b_0\pmb{Y}^{(p)} + b_1 \pmb{X}^{(p)} + \cdots + b_p \pmb{X}_p \ge 0$.
The identity in law (9) implies that the equality occurs with probability $0$, which proves \eqref{eq:star2}.

\subsection{Proof of Proposition \ref{prop2}} \label{B2}
Let $(X_1, \ldots, X_n)$ be i.i.d. samples with density $f_{X^{(p)}}$.
It is well known that the distribution of the order statistics $(X_{(1)}, \ldots, X_{(n)})$ is given by 
$n! \, \prod_{i=1}^n f_{X^{(p)}}(x_i)$ for $x_1 < \cdots < x_n$.
Note that there exists $t \in \mathbb{R}$ separating $X_{(1)} < \cdots < X_{(n)}$ if and only if for some $k \in \{0, \ldots, n\}$, 
the responses corresponding to $X_{(1)}, \ldots, X_{(k)}$ is of the same sign, and those corresponding to $X_{(k+1)}, \ldots, X_{(n)}$ is of the opposite sign.
Consequently,
\begin{align*}
& \mathbb{P}(\exists t: \mbox{separate } X_1, \ldots, X_n)  \\
&  = \int_{x_1 < \cdots < x_n} \mathbb{P}(\exists t \mbox{ separate } X_1, \ldots X_n| X_{(i)} = x_i  \, \forall i) \cdot n! \,  \prod_{i=1}^n f_{X^{(p)}}(x_i)dx_1 \ldots dx_n \\
& = \int_{x_1 < \cdots < x_n} n! \,  \prod_{i=1}^n f_{X^{(p)}}(x_i) \Bigg( \prod_{i=1}^n  p_{-}(x_i) + \sum_{k=1}^{n-1} \prod_{i=1}^k p_{-}(x_i) \prod_{i = k+1}^n p_{+}(x_i) \\
&  \qquad  \qquad \qquad \qquad  \qquad + \sum_{k=1}^{n-1} \prod_{i=1}^k p_{+}(x_i) \prod_{i = k+1}^n p_{-}(x_i) + \prod_{i=1}^n  p_{+}(x_i) \Bigg) dx_1 \ldots dx_n.
\end{align*}

Moreover, $\int_{x_1 < \cdots < x_n} n! \, \prod_{i=1}^n f_{X^{(p)}}(x_i) \prod_{i=1}^n p_{\pm}(x_i) dx_1\ldots dx_n$ is equal to 
\begin{equation*}
\left(\int_{\mathbb{R}} f_{X^{(p)}}(x) p_{\pm}(x)dx\right)^n = (\mathbb{E}p_{\pm}(X^{(p)}) )^n.
\end{equation*}
For $1 \le k \le n-1$, 
\begin{align*}
& \int_{x_1 < \cdots < x_n} n! \, \prod_{i=1}^n f_{X^{(p)}}(x_i) \prod_{i=1}^k p_{-}(x_i) \prod_{i = k+1}^n p_{+}(x_i) dx_1\ldots dx_n  \\
& = n! \int_{x_{k+1} = - \infty}^{\infty} f_{X^{(p)}}(x_{k+1}) p_{+}(x_{k+1}) dx_{k+1} \\
& \qquad \qquad \qquad \qquad \qquad  \left( \int_{x_1 < \cdots < x_{k+1}} \prod_{i=1}^k f_{X^{(p)}}(x_i)  p_{-}(x_i) dx_1 \ldots dx_k \right) \\
& \qquad \qquad \qquad \qquad \qquad  \left( \int_{x_n > \cdots > x_{k+1} } \prod_{i=k+2}^n f_{X^{(p)}}(x_i)  p_{+}(x_i) dx_{k+2} \ldots dx_n  \right)  \\
& = \frac{n!}{k! (n-k-1)!} \int_{x_{k+1} = - \infty}^{\infty} f_{X^{(p)}}(x_{k+1})p_{+}(x_{k+1})^k p_{-}(x_{k+1}) \overline{G}^{n-k-1}_{p,+}(x_{k+1}) dx_{k+1} \\
& = n \binom{n-1}{k} \mathbb{E}[p_{+}(X^{(p)}) G^k_{p,-}(X^{(p)}) \overline{G}^{n-k-1}_{p,+}(X^{(p)})].
\end{align*}
Combining the above identities yields
\begin{align}
& \mathbb{P}(\exists t \mbox{ separate } X_1, \ldots, X_n) \notag \\ 
& = (\mathbb{E}p_{-}(X^{(p)}) )^n + n \sum_{k=1}^n \binom{n-1}{k} \mathbb{E}[p_{+}(X^{(p)}) G^k_{p,-}(X^{(p)}) \overline{G}^{n-k-1}_{p,+}(X^{(p)})] \notag \\
& \qquad   + n \sum_{k=1}^n \binom{n-1}{k} \mathbb{E}[p_{-}(X^{(p)}) G^k_{p,+}(X^{(p)}) \overline{G}^{n-k-1}_{p,-}(X^{(p)})] + (\mathbb{E}p_{+}(X^{(p)}) )^n \notag \\
& = (\mathbb{E}p_{-}(X^{(p)}) )^n + n \mathbb{E}[p_{+}(X^{(p)}) (G_{p,-}(X^{(p)}) + \overline{G}_{p,+}(X^{(p)}))^{n-1}]  \notag \\
&  \qquad   +  n \mathbb{E}[p_{-}(X^{(p)}) (G_{p,+}(X^{(p)}) + \overline{G}_{p,-}(X^{(p)})^{n-1}] + (\mathbb{E}p_{+}(X^{(p)}) )^n.   \label{eq:sstd}
\end{align}
Finally, the condition \eqref{eq:pG} together with \eqref{eq:sstd} lead to the desired result.

\subsection{Proof of Theorem \ref{thm:2}} \label{B3}
We start with the proof of Lemma \ref{lem:sa}.

\begin{proof}[Proof of Lemma \ref{lem:sa}]
By law of large numbers of triangular arrays,  the condition $\sup_i \mathbb{E}g^4(\pmb{\lambda}, \pmb{\xi}_{i,n}) < \infty$ implies that $\frac{1}{n} \sum_{i = 1}^n g(\pmb{\lambda}, \pmb{\xi}_{i,n}) - \mathbb{E} g(\pmb{\lambda}, \pmb{\xi}_n) \rightarrow 0$ a.s. 
It follows from \cite[Theorem 2.3]{AR95} that
\begin{equation*}
\frac{1}{n} \sum_{i = 1}^n g(\pmb{\lambda}, \pmb{\xi}_{i,n}) - \mathbb{E} g(\pmb{\lambda}, \pmb{\xi}_n)\mbox{ epi-converges to } 0 \quad a.s.
\end{equation*}
It is well known that if a sequence of convex functions converge pointwise, then they converge uniformly on compact sets.
Since $\mathbb{E}g(\pmb{\lambda}, \pmb{\xi}_n) \rightarrow \mathbb{E}g(\pmb{\lambda}, \pmb{\xi})$ for all $\pmb{\lambda}$ and $g(\cdot, \pmb{\xi})$ is convex, the convergence is uniform on compact sets.
This implies the epi-convergence. 
Therefore,
\begin{equation*}
\frac{1}{n} \sum_{i = 1}^n g(\pmb{\lambda}, \pmb{\xi}_{i,n})  \mbox{ epi-converges to } \mathbb{E} g(\pmb{\lambda}, \pmb{\xi}) \quad a.s.
\end{equation*}
Combining with \cite[Proposition 3.3]{DW88} yields the desired result.
\end{proof}

Now we are ready to prove Theorem \ref{thm:2}. 
We specialize to $\pmb{\lambda} = (\lambda_0, \lambda_1)$, $\pmb{\xi}_n= (Y^{(p)},X^{(p)},Z)$ with $(Y^{(p)},X^{(p)}) \sim F_{\alpha_0,\beta_0,\gamma_0}$, $\pmb{\xi} = (Y, X, Z)$, 
and $Z \sim \mathcal{N}(0,1)$ independent of $\{(Y^{(p)},X^{(p)}), (Y,X) \}$, and 
\begin{equation}
\label{eq:fctg}
g(\pmb{\lambda}, \pmb{\xi}) = (\lambda_0 Y + \lambda_1 X - Z)_{+}^2.
\end{equation}
It is clear that the function $g$ defined by \eqref{eq:fctg} is measurable and non-negative, and $g(\cdot, \pmb{\xi})$ is convex. 
It follows from $\sup_p \mathbb{E}[(X^{(p)})^8] < \infty$ that $\sup_n g^4(\pmb{\lambda}, \pmb{\xi}_n) < \infty$.
By Assumption \ref{assumpt3}, $\mathbb{E}[(X^{(p)})^2]$ converges to $\alpha_0^2$, and by Assumption \ref{assumpt4}, $(Y^{(p)}, X^{(p)})$ converges in distribution to $(Y,X)$.
Now by \cite[Lemma 8.3]{BF81}, we get $\mathbb{E}g(\pmb{\lambda}, \pmb{\xi}_n) \rightarrow \mathbb{E}g(\pmb{\lambda}, \pmb{\xi})$ for all $\pmb{\lambda}$.

Let $\pmb{\lambda}_{\min}$ be a minimum of $\mathbb{E}g(\pmb{\lambda}, \pmb{\xi})$.
By convexity of $\pmb{\lambda} \rightarrow \mathbb{E}g(\pmb{\lambda}, \pmb{\xi})$, there exists $r > 0$ such that
$\mathbb{E}g(\pmb{\lambda}_{\min}, \pmb{\xi}) < \min_{r \le |\pmb{\lambda}| \le r + 1} \mathbb{E}g(\pmb{\lambda}, \pmb{\xi})$, and 
$\min_{\pmb{\lambda}} \mathbb{E}g(\pmb{\lambda}, \pmb{\xi}) = \min_{|\pmb{\lambda}| \le r + 1} \mathbb{E}g(\pmb{\lambda}, \pmb{\xi})$.
Note that $\mathbb{E}g(\pmb{\lambda}, \pmb{\xi}_n)$ converges uniformly to $\mathbb{E}g(\pmb{\lambda}, \pmb{\xi})$ on $\{\pmb{\lambda}: r \le |\pmb{\lambda}| \le r+1\}$.
So for $p$ large enough, 
$\mathbb{E}g(\pmb{\lambda}_{\min}, \pmb{\xi}_n) < \min_{r \le |\pmb{\lambda}| \le r+1} \mathbb{E}g(\pmb{\lambda}, \pmb{\xi}_n)$ and
$\min_{\pmb{\lambda}} \mathbb{E}g(\pmb{\lambda}, \pmb{\xi}_n) = \min_{r \le |\pmb{\lambda}| \le r+1} \mathbb{E}g(\pmb{\lambda}, \pmb{\xi}_n)$.
Now by \cite[Theorem 2.1]{KS83}, we have as $p \rightarrow \infty$,
\begin{equation*}
   \min_{\pmb{\lambda}}\mathbb{E}g(\pmb{\lambda}, \pmb{\xi}_n) = \min_{|\pmb{\lambda}| \le r + 1}\mathbb{E}g(\pmb{\lambda}, \pmb{\xi}_n) \rightarrow \min_{|\pmb{\lambda}| \le r + 1}\mathbb{E}g(\pmb{\lambda}, \pmb{\xi}) = \min_{\pmb{\lambda}}\mathbb{E}g(\pmb{\lambda}, \pmb{\xi}).
\end{equation*}

By Lemma \ref{lem:sa}, it suffices to prove that the set of minimizers $\argmin \widehat{G}_n$ is non-empty and bounded in probability. 
We aim to show that under the assumptions in Theorem \ref{thm:main}, the problem $(13)$ has a unique minimizer $\pmb{\lambda}_0$, and for any minimizer $\widehat{\pmb{\lambda}}_n \in \argmin \widehat{G}_n$, $|\widehat{\pmb{\lambda}}_n - \pmb{\lambda}_0| = \mathcal{O}_P(1)$.

We prove these statements in the next two lemmas.
It is easily seen that $G(\pmb{\lambda})$ and $\widehat{G}_n(\pmb{\lambda})$ are convex.
The following lemma shows that the function $G$ is strongly convex, which was stated in \cite{CS18} without proof.
Here we give a complete proof.

\begin{lemma}
Under the assumptions in Theorem \ref{thm:main}, the function $\pmb{\lambda} \rightarrow G(\pmb{\lambda})$ with $g(\cdot)$ defined by \eqref{eq:fctg} is strongly convex. That is, there exists $\alpha_1 > \alpha_0 > 0$ such that 
\begin{equation}
\label{eq:SCV}
\alpha_0 \pmb{I}_2 \preceq \nabla^2 G(\pmb{\lambda}) \preceq \alpha_1 \pmb{I}_2.
\end{equation}
\end{lemma}

\begin{proof}
Elementary analysis shows that
\begin{align*}
\nabla^2 G(\pmb{\lambda}) &=
\begin{pmatrix}
    \mathbb{E}[Y^2 \Phi(\lambda_0 Y + \lambda_1 X)] & \mathbb{E}[YX \Phi(\lambda_0 Y + \lambda_1 X)] \\
   \mathbb{E}[YX \Phi(\lambda_0 Y + \lambda_1 X)] & \mathbb{E}[X^2 \Phi(\lambda_0 Y + \lambda_1 X)]
\end{pmatrix} \\
& = 
\begin{pmatrix}
    \mathbb{E}[ \Phi(\lambda_0 V + \lambda_1 VU)] & \mathbb{E}[U \Phi(\lambda_0 V + \lambda_1 VU)] \\
   \mathbb{E}[U \Phi(\lambda_0 V + \lambda_1 VU)] & \mathbb{E}[U^2 \Phi(\lambda_0 Y + \lambda_1 VU)]
\end{pmatrix},
\end{align*}
where $\Phi(\cdot)$ is the CDF of standard normal, $U$ is defined in Assumption \ref{assumpt4}, and $\mathbb{P}(V=1|U) = 1 - \mathbb{P}(V = -1|U) = p_{+}(U)$.
The r.h.s. of \eqref{eq:SCV} is clear. 
By Cauchy-Schwarz inequality, $\det \nabla^2 G(\pmb{\lambda}) \ge 0$. 
If $\det \nabla^2 G(\pmb{\lambda}) = 0$, then $U$ is constant almost surely which violates the non-degeneracy of $U$.
Thus, $G(\pmb{\lambda})$ is strictly convex.

Note that 
$$\mathbb{E}[ \Phi(\lambda_0 V + \lambda_1 VU)] = \mathbb{E}[\Phi(\lambda_0 + \lambda_1 X) p_{+}(X)] + \mathbb{E}[\Phi(-\lambda_0 - \lambda_1 X) p_{-}(X)],$$ 
and the decomposition holds for other terms. So
\begin{align}
\label{eq:decomp}
\nabla^2 G(\pmb{\lambda}) &=
\begin{pmatrix}
    \mathbb{E}[\Phi(\lambda_0  + \lambda_1 U) p_{+}(U)] & \mathbb{E}[U \Phi(\lambda_0  + \lambda_1 U) p_{+}(U)] \\
   \mathbb{E}[U \Phi(\lambda_0  + \lambda_1 U) p_{+}(U)] & \mathbb{E}[U^2 \Phi(\lambda_0  + \lambda_1 U) p_{+}(U)]
\end{pmatrix} \notag \\
& \qquad \qquad \qquad + 
\begin{pmatrix}
    \mathbb{E}[\Phi(\lambda_0  + \lambda_1 U) p_{-}(U)] & \mathbb{E}[U \Phi(\lambda_0  + \lambda_1 U) p_{-}(U)] \\
   \mathbb{E}[U \Phi(\lambda_0  + \lambda_1 U) p_{-}(U)] & \mathbb{E}[U^2 \Phi(\lambda_0  + \lambda_1 U) p_{-}(U)]
\end{pmatrix}. 
\end{align}
Without loss of generality, consider $\lambda_1, \lambda_2 > 0$. For $\lambda_0, \lambda_1$ sufficiently large, 
\begin{itemize}[itemsep = 5 pt]
\item
  if $\lambda_0/\lambda_1$ is large, then $\mathbb{E}[\Phi(\lambda_0  + \lambda_1 U) p_{+}(U)]$ can be approximated by $\mathbb{E}p_{+}(U)$.
\item
  if $\lambda_0/\lambda_1$ is small, then $\mathbb{E}[\Phi(\lambda_0  + \lambda_1 U) p_{+}(U)]$ can be approximated by $\mathbb{E}[1(U > -\lambda_0/\lambda_1 + \eta)p_{+}(U)]$, where $\eta$ is a fixed small value.
\end{itemize}
The approximation also holds for other terms.
By the strict convexity and the approximation, we can show that there exist $N_{+} > 0$ such that for $\lambda_0, \lambda_1 > N_{+}$,
\begin{equation*}
\begin{pmatrix}
    \mathbb{E}[\Phi(\lambda_0  + \lambda_1 U) p_{+}(U)] & \mathbb{E}[U \Phi(\lambda_0  + \lambda_1 U) p_{+}(U)] \\
   \mathbb{E}[U \Phi(\lambda_0  + \lambda_1 U) p_{+}(U)] & \mathbb{E}[U^2 \Phi(\lambda_0  + \lambda_1 U) p_{+}(U)]
\end{pmatrix}
\succeq \epsilon_{+} \pmb{I}_2,
\end{equation*}
for some $\epsilon_{+} > 0$.
Similarly, there exist $N_{-} > 0$ such that for $\lambda_0, \lambda_1 > N_{-}$, we get a bound $\epsilon_{-}\pmb{I}_2$ for the second term on the r.h.s. of \eqref{eq:decomp}.
Thus, for $\lambda_1, \lambda_2 > \max(N_{+}, N_{-})$, $\nabla^2 G(\pmb{\lambda}) \succeq \min(\epsilon_{+}, \epsilon_{-}) \pmb{I}_2$. 
By continuity of $\det \nabla^2 G(\pmb{\lambda})$, we get $\nabla^2 G(\pmb{\lambda}) \succeq \epsilon' \pmb{I}_2$ for $\lambda_0, \lambda_1 < \max(N_{+}, N_{-})$. 
It suffices to take $\alpha_0 = \min(\epsilon_{+}, \epsilon_{-}, \epsilon')$ to conclude.
\end{proof}

Finally, we prove that the set of minimizers $\argmin \, \widehat{G}_n$ is bounded in probability. 
The argument can be used to show that $|\widehat{\pmb{\lambda}}_n -  \pmb{\lambda}_0| = \mathcal{O}_P(n^{-1/4})$, where $\widehat{\pmb{\lambda}}_n$ is any minimizer of $\widehat{G}_n$.
The proof is adapted from \cite[Lemma 4]{CS18}, which we include for completeness.
\begin{lemma}
Under the assumptions in Theorem \ref{thm:main}, we have $|\widehat{\pmb{\lambda}}_n - \pmb{\lambda}_0| = \mathcal{O}_P(1)$, where $\widehat{\pmb{\lambda}}_n \in \argmin \, \widehat{G}_n$.
\end{lemma}

\begin{proof}
For any $\pmb{\lambda} \in \mathbb{R}^2$, the strong convexity \eqref{eq:SCV} gives that
\begin{equation*}
G(\pmb{\lambda}) \ge G(\pmb{\lambda}_0) + \frac{\alpha_0}{2}|\pmb{\lambda} - \pmb{\lambda}_{0}|^2.
\end{equation*}

Fix $x \ge 1$. For any $\pmb{\lambda}$ on the circle $C(x): = \{\pmb{\lambda}: |\pmb{\lambda} -\pmb{\lambda}_{0}| = x\}$, we have
\begin{equation}
\label{eq:incond}
G(\pmb{\lambda}) \ge G(\pmb{\lambda}_{0}) + 3y, \quad y := \frac{\alpha_0 x^2}{6}.
\end{equation}

Fix $z = G(\pmb{\lambda}_{0}) + y$, and consider the event
\begin{equation*}
E: = \left\{   \widehat{G}_n(\pmb{\lambda}_0) < z \mbox{ and } \inf_{\pmb{\lambda} \in C(x)} \widehat{G}_n(\pmb{\lambda}) > z   \right\}.
\end{equation*}

By convexity of $\widehat{G}_n$, when $E$ occurs, $\widehat{\pmb{\lambda}}_{n}$ must lie in the circle. 
Hence, $|\widehat{\pmb{\lambda}}_n - \pmb{\lambda}_0| \le x$.

Next we prove that the event $E$ occurs with high probability.
Fix $d$ equi-spaced point $\{\pmb{\lambda}_i\}_{i = 1}^d$ on the set $C(x)$. 
Take any point $\pmb{\lambda}$ on the circle, and let $\pmb{\lambda}_i$ be its closest point.
So $|\pmb{\lambda} - \pmb{\lambda}_i| \le \pi x/d$.
By convexity of $\widehat{G}_n$,
\begin{equation}
\label{eq:lb}
\widehat{G}_n(\pmb{\lambda}) \ge \widehat{G}_n(\pmb{\lambda}_i) - |\nabla \widehat{G}(\pmb{\lambda}_i)| |\pmb{\lambda} - \pmb{\lambda}_i|.
\end{equation}
Define the event
\begin{equation*}
B: = \left\{\max_i  |\nabla \widehat{G}_n(\pmb{\lambda}_i) - \nabla G(\pmb{\lambda}_i) | \le x   \right\}.
\end{equation*}
By Chebyshev inequality and union bound,  we get
\begin{equation}
\label{eq:boundB}
\mathbb{P}(B^c) \le \frac{d \sigma^2}{n x^2},
\end{equation}
where $\sigma^2 : = \sup_n \Var[g(\pmb{\lambda}, \pmb{\xi}_n)] < \infty$.
As $|\nabla^2 G|$ is bounded and $\nabla G(\pmb{\lambda}_{0}) = 0$, 
\begin{equation*}
|\nabla G(\pmb{\lambda}_i )| \le \alpha_1 |\pmb{\lambda}_i - \pmb{\lambda}_{0}| = \alpha_1 x.
\end{equation*}

Now for $n$ sufficiently large, on $B$ we have
$|\nabla \widehat{G}_n(\pmb{\lambda}_i)| |\pmb{\lambda} - \pmb{\lambda}_i| \le \pi (1+\alpha_1)x^2/d$.
Choose $d > 6 \pi (1+\alpha_1)/\alpha_0$ so that 
$|\nabla \widehat{G}_n(\pmb{\lambda}_i)| |\pmb{\lambda} - \pmb{\lambda}_i| \le y$.
Then it follows from \eqref{eq:lb} that on $B$,
\begin{equation*}
\inf_{\pmb{\lambda} \in C(x)} \widehat{G}_n(\pmb{\lambda}) \ge \min_i \widehat{G}(\pmb{\lambda}_i) - y.
\end{equation*}
Observe that
\begin{equation*}
\widehat{G}_n(\pmb{\lambda}_i ) > G(\pmb{\lambda}_i) - y \Longrightarrow \widehat{G}_n(\pmb{\lambda}_i) - y > G(\pmb{\lambda}_{0}) + y = z,
\end{equation*}
since $G(\pmb{\lambda}_i) \ge G(\pmb{\lambda}_{0}) + 3y$ by \eqref{eq:incond}.
Consequently,
\begin{equation*}
\mathbb{P}(E^c) \le \mathbb{P}(B^c) + \mathbb{P}(\widehat{G}_n(\pmb{\lambda}_0) \ge G(\pmb{\lambda}_{0}) + y) + \sum_{i = 1}^d \mathbb{P}(\widehat{G}_n(\pmb{\lambda}_i ) \le G(\pmb{\lambda}_i) - y),
\end{equation*}
where $\mathbb{P}(B^c)$ is controlled by \eqref{eq:boundB}, and the last two terms can also be bounded by Chebyshev inequality.
\end{proof}

\section{Conclusion}
\label{s5}
In this paper, we proved a phase transition for the existence of the maximum likelihood estimate in high-dimensional generalized linear models 
with elliptical covariates. 
We derived an explicit formula for the phase transition boundary, depending on the regression coefficients and the scaling parameter of the covariate
distribution. 
Our result extends a previous one in \cite{CS18}, and elucidates a rich structure in the phase transition phenomenon.
We believe that the phase transition also holds for multinomial response models such as the Poisson regression and the log-linear regression. 
See \cite{CM08, FR12} for further discussions.
We hope that this work will trigger further research towards a theory of hypothesis testing for generalized linear models with non-Gaussian covariates.

\section*{Acknowledgements}
Wenpin Tang gratefully acknowledges financial support through a startup grant at Columbia University.

\appendix
\section{Empirical Results}\label{appendix:results}

To study whether the MLE exists for a given pair of $(\gamma_0, \kappa)$ with some distribution for $R$, we generate the data by the mechanism described in Section \ref{s3}.
We fix the sample size at $n = 1000$, and vary $\gamma_0 \in \{0.01, 0.02, \ldots, 10.00\}$ and $\kappa = p/n \in \{0.005, 0.01, \ldots, 0.6\}$.
We generate $100$ such datasets, and for each dataset we solve the linear programming \eqref{eq:separable_lp} by checking whether there is a nonzero solution using the package \textit{CVXOPT}\footnote{https://cvxopt.org/documentation/index.html} in python.   

On the other hand, the optimization problem \eqref{eq:hmle} does not have a closed-form solution. 
Here we solve numerically this convex optimization. 
Using the same data generation mechanism but with a sample size $n = 4000$, we compute $h_{\text{MLE}}$ by \textit{CVXOPT} as well.
We repeat the procedure for $100$ times and take the average of these replicates as the reported $h_{\text{MLE}}$. 
Here we take $\gamma_0 \in \{0.5, 1.0, \ldots, 10.00\}$ and $\kappa = p/n \in \{0.02, 0.04, \ldots, 0.6\}$.

\begin{figure}[ht]
  \centering
  \begin{minipage}{\textwidth}
    \centering
    \includegraphics[width=  0.9\textwidth]{pt_link.png}
    \subcaption{$\hat{\mathbb{P}} (\text{MLE}~exists)$ for $\pmb{A} = \pmb{I}_p$}
  \end{minipage}\\
  \begin{minipage}{\textwidth}
    \centering
    \includegraphics[width=  0.9\textwidth]{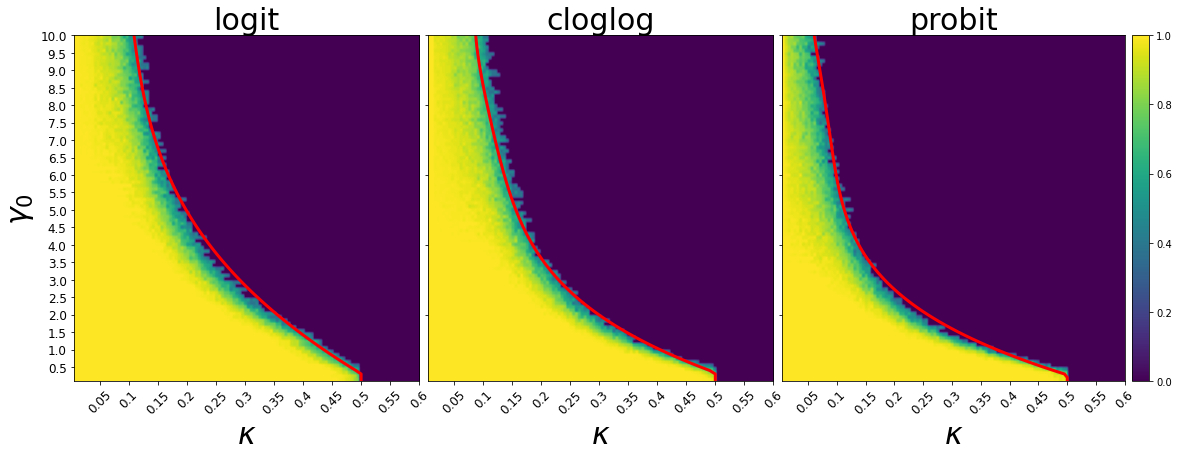}
    \subcaption{$\hat{\mathbb{P}} (\text{MLE}~exists)$ for $\pmb{A} \neq \pmb{I}_p$}
  \end{minipage}\\
  \begin{minipage}{\textwidth}
    \centering
    \includegraphics[width= 0.9\textwidth]{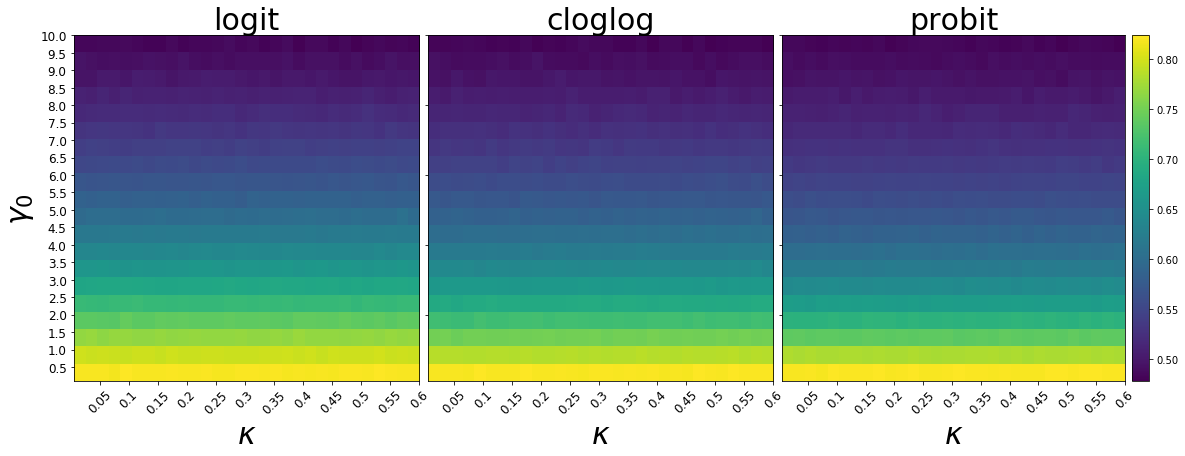}
    \subcaption{$h_{\text{MLE}}$}
  \end{minipage}
  \caption{Phase transition of the MLE existence for $R \sim $ chi distribution ($df = p$) with different link functions. Upper: The value of each grid in the heatmap is the proportion of times that the MLE exists across the $100$ replicates. The red curve is the theoretical $h_{\text{MLE}}$ boundary given by \eqref{eq:hmle}. Middle: The same setup as the upper figures, but $\pmb{A}$ is generated so that each entry $A_{ij}$ is sampled from the standard Gaussian distribution. To make sure it is full rank, we let $\pmb{A} \leftarrow \frac{1}{10000} \Vert \pmb{A} \Vert_2 \pmb{I}_p + \pmb{A}$. Bottom: The heatmap for the $h_{\text{MLE}}$. Each grid is the averaged value $h_{\text{MLE}}$ across $100$ replicates.}
  \label{fig:pt_chip_link_full}  
\end{figure}

\begin{figure}[ht]
  \centering
  \begin{minipage}{\textwidth}
    \centering
    \includegraphics[width=\textwidth]{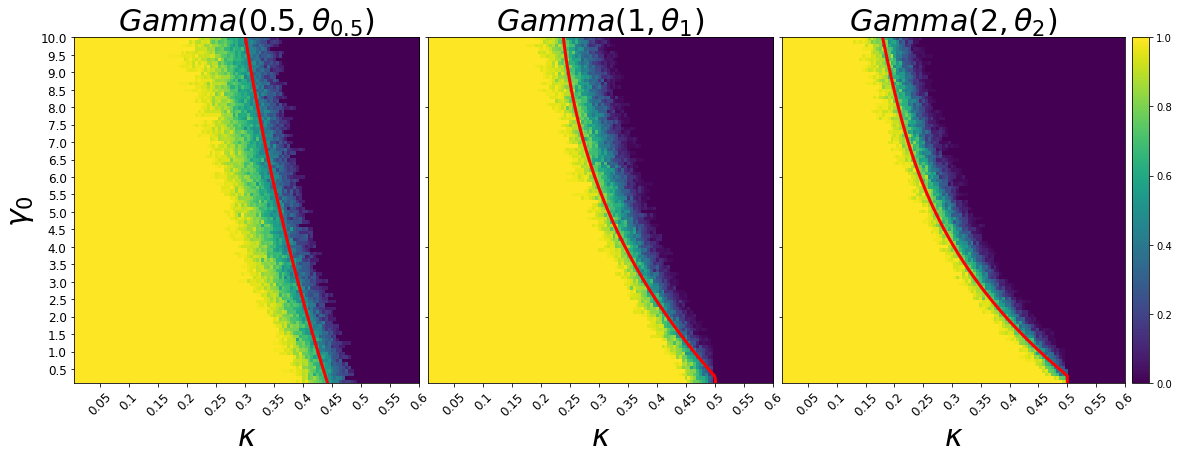}
    \subcaption{$\hat{\mathbb{P}} (\text{MLE}~exists)$}
  \end{minipage}\\
  \begin{minipage}{\textwidth}
    \centering
    \includegraphics[width=\textwidth]{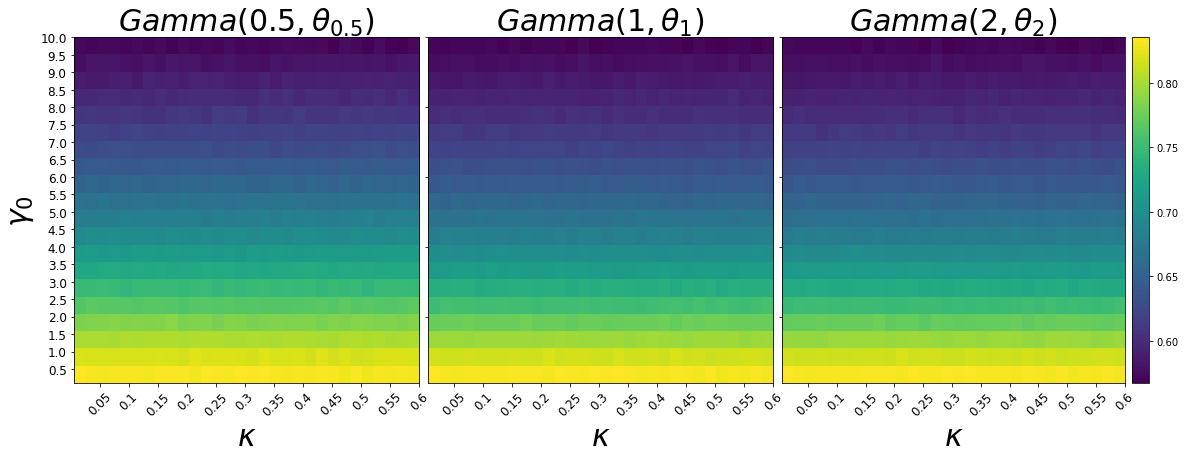}
    \subcaption{$h_{\text{MLE}}$}
  \end{minipage}
  \caption{Phase transition of the MLE existence for $R \sim$ Gamma distributions with different parameters and the logit link (see Section \ref{sec:sim_gamma}). Upper: The value of each grid in the heatmap is the proportion of times that the MLE exists across the $100$ replicates. The red curve is the theoretical $h_{\text{MLE}}$ boundary given by \eqref{eq:hmle}. Bottom: The heatmap for the $h_{\text{MLE}}$. Each grid is the averaged value $h_{\text{MLE}}$ across $100$ replicates.}
  \label{fig:pt_gamma_full}  
\end{figure}

\begin{figure}[ht]
  \centering
  \begin{minipage}{\textwidth}
    \centering
    \includegraphics[width=\textwidth]{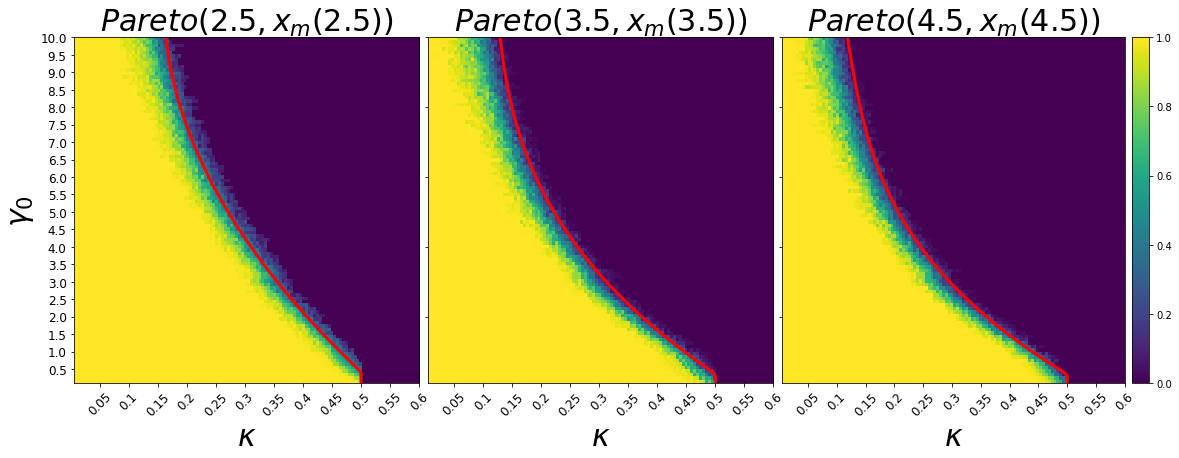}
    \subcaption{$\hat{\mathbb{P}} (\text{MLE}~exists)$}
  \end{minipage}\\
  \begin{minipage}{\textwidth}
    \centering
    \includegraphics[width=\textwidth]{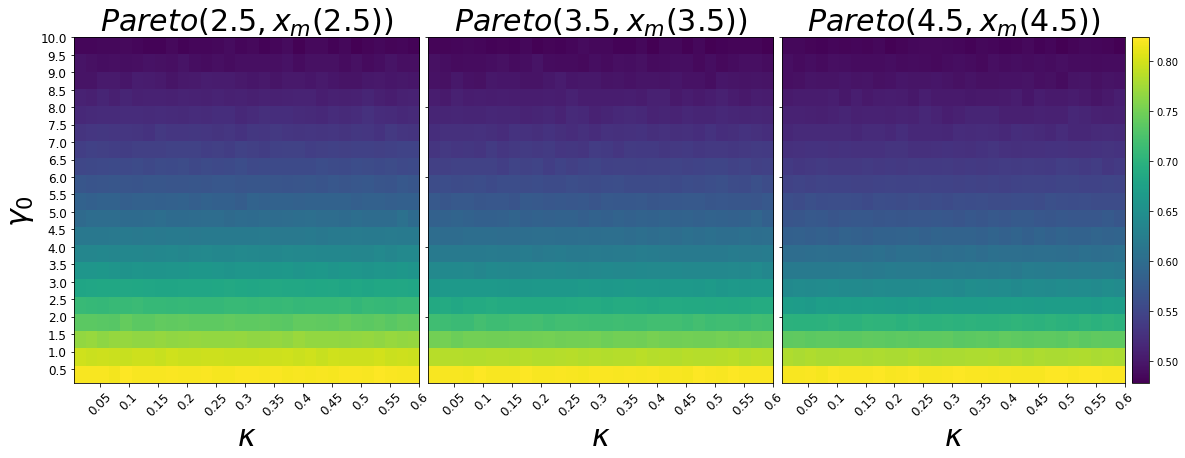}
    \subcaption{$h_{\text{MLE}}$}
  \end{minipage}
  \caption{Phase transition of the MLE existence for the Pareto distributions with different parameters and the logit link (see Section \ref{sec:sim_moment_tail}). Upper: The value of each grid in the heatmap is the proportion of times that the MLE exists across the $100$ replicates. The red curve is the theoretical $h_{\text{MLE}}$ boundary given by \eqref{eq:hmle}. Bottom: The heatmap for the $h_{\text{MLE}}$. Each grid is the averaged value $h_{\text{MLE}}$ across $100$ replicates.}
  \label{fig:pt_pareto_full}  
\end{figure}

\begin{figure}[ht]
  \centering
  \includegraphics[width=\textwidth]{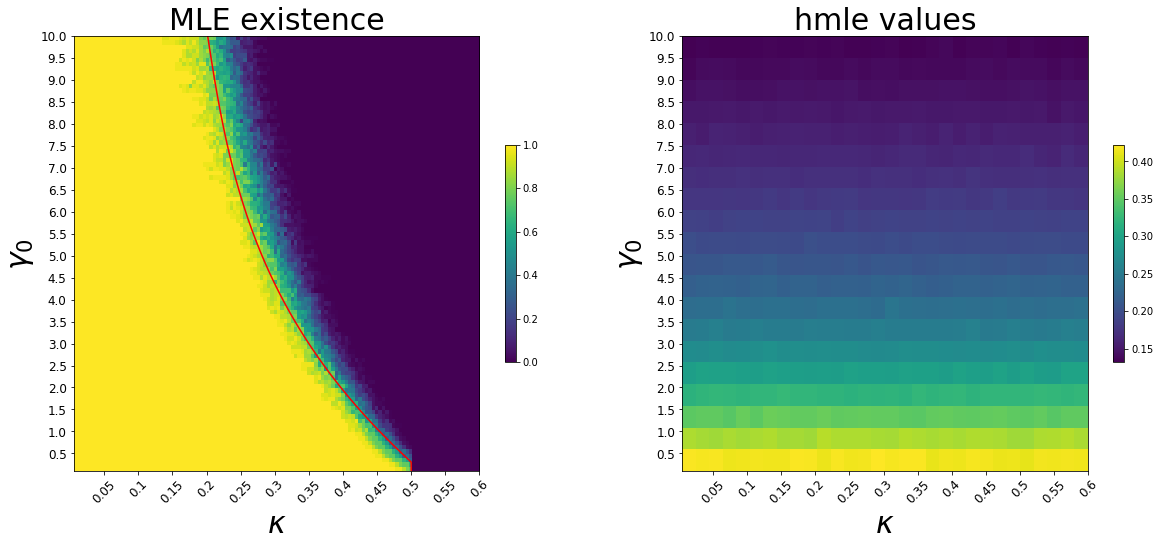}
  \caption{Phase transition of the MLE existence for the half-normal distribution and the logit link function. Left: The value of each grid in the heatmap is the proportion of times that the MLE exists across the $100$ replicates. The red curve is the theoretical $h_{\text{MLE}}$ boundary given by \eqref{eq:hmle}. Right: The heatmap for the $h_{\text{MLE}}$. Each grid is the averaged value $h_{\text{MLE}}$ across $100$ replicates.}
  \label{fig:pt_half_normal_full}  
\end{figure}

\begin{figure}[ht]
  \centering
  \includegraphics[width=\textwidth]{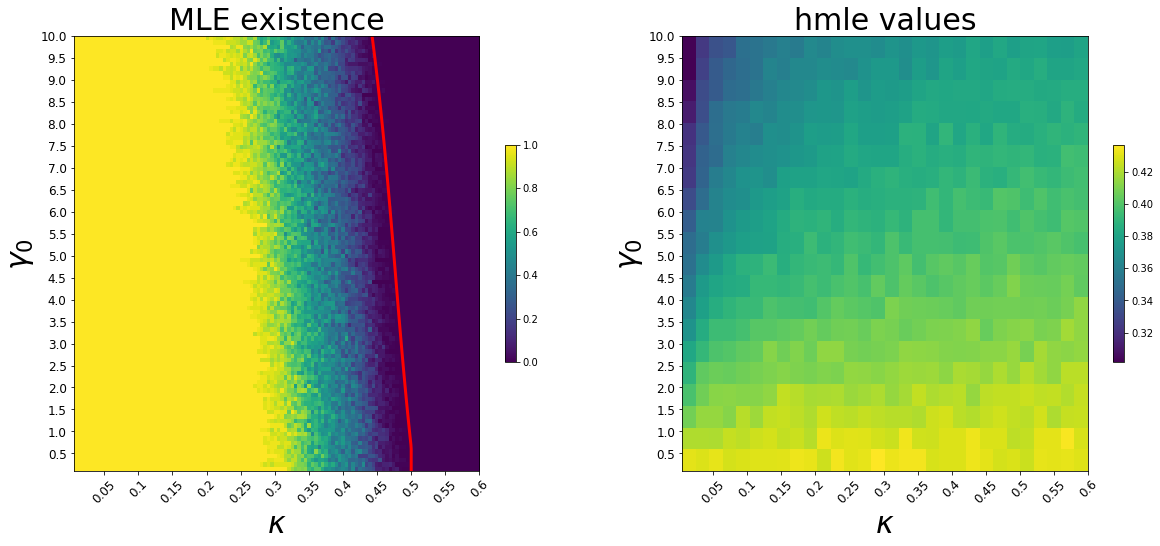}
  \caption{Phase transition of the MLE existence for the log-Normal distribution and the logit link function. Left: The value of each grid in the heatmap is the proportion of times that the MLE exists across the $100$ replicates. The red curve is the theoretical $h_{\text{MLE}}$ boundary given by \eqref{eq:hmle}. Right: The heatmap for the $h_{\text{MLE}}$. Each grid is the averaged value $h_{\text{MLE}}$ across $100$ replicates.}
  \label{fig:pt_log_normal_full}  
\end{figure}

\clearpage

\bibliographystyle{plain}
\bibliography{unique}

\end{document}